\documentclass{amsart}
\usepackage{amssymb,latexsym,amsmath}
\usepackage{amscd,amsthm}
\usepackage{tikz-cd}

\usepackage[all]{xy}

\newtheorem{theorem}{Theorem}[section]

\newtheorem{them}{Theorem}

\newtheorem{lemma}[theorem]{Lemma}
\newtheorem{proposition}[theorem]{Proposition}
\newtheorem{corollary}[theorem]{Corollary}

\theoremstyle{definition}
\newtheorem{definition}[theorem]{Definition}

\newtheorem{setup}[theorem]{Setup}
\newtheorem{remark}{Remark}

\DeclareMathOperator{\Ext}{Ext}
\DeclareMathOperator{\GExt}{G-Ext}

\DeclareMathOperator{\Hom}{Hom}

\DeclareMathOperator{\cok}{cok}
\DeclareMathOperator{\im}{Im}

\newcommand{\cat}[1]{\mathcal{#1}}           %% font for categories

\newcommand{\tensor}{\otimes}

\newcommand{\class}[1]{\mathcal{#1}}   %% font for classes
\newcommand{\N}{\mathbb{N}}
\newcommand{\Z}{\mathbb{Z}}

\newcommand{\mathcolon}{\colon\,} %% Hovey uses for maps, like f: A -> B

\newcommand{\ch}{\textnormal{Ch}(R)}

\newcommand{\cha}[1]{\textnormal{Ch}(\mathcal{#1})}

\newcommand{\rmod}{R\text{-Mod}}

\newcommand{\qcor}{\textnormal{Qco}(R)}

\newcommand{\rightperp}[1]{#1^{\perp}}
\newcommand{\leftperp}[1]{{}^\perp #1}

\newcommand{\homcomplex}{\mathit{Hom}}

\newcommand{\Ho}[1]{{\textnormal{Ho}(#1)}}

\begin{document}

\title{The derived category with respect to a generator}

\author{James Gillespie}
\address{Ramapo College of New Jersey \\
         School of Theoretical and Applied Science \\
         505 Ramapo Valley Road \\
         Mahwah, NJ 07430}
\email[Jim Gillespie]{jgillesp@ramapo.edu}
\urladdr{http://pages.ramapo.edu/~jgillesp/}

\date{\today}

\begin{abstract}
Let $\cat{G}$ be any Grothendieck category along with a choice of generator $G$, or equivalently a generating set $\{G_i\}$. We introduce the derived category $\class{D}(G)$, which kills all $G$-acyclic complexes, by putting a suitable model structure on the category $\cha{G}$ of chain complexes. It follows that the category $\class{D}(G)$ is always a well-generated triangulated category. It is compactly generated whenever the generating set $\{G_i\}$ has each $G_i$ finitely presented, and in this case we show that two recollement situations hold. The first is when passing from the homotopy category $K(\cat{G})$ to $\class{D}(G)$. The second is a $G$-derived analog of a recollement due to Krause. We describe several examples ranging from pure and clean derived categories to quasi-coherent sheaves on the projective line.
\end{abstract}

\maketitle

\section{Introduction}\label{sec-intro}

This paper is about doing homological algebra with respect to a given generator in a Grothendieck category.
Let $R$ be a ring and $\ch$ denote the category of chain complexes of (left) $R$-modules. Recall that the usual derived category $\class{D}(R)$ is defined by first constructing the homotopy category $K(R)$ of unbounded chain complexes of $R$-modules, and then formally inverting the homology isomorphisms. $R$ itself, when viewed as an $R$-module is a generator for $\rmod$. But when $R$ is viewed as a chain complex in degree zero, it is a weak generator for $\class{D}(R)$ which essentially means it can detect exactness. Note that for a chain complex $X$, the standard isomorphism $\Hom_R(R,X) \cong X$ allows one to view the homology of $X$ as $H_n[\Hom_R(R,X)]$. Similarly, homology isomorphisms can be viewed as those chain maps $X \xrightarrow{} Y$ in $\ch$ which become homology isomorphisms after applying $\Hom_R(R,-)$.

But sometimes the derived category $\class{D}(R)$ is not the right home for the homological algebra one is interested in. For example, there is the pure derived category of a ring $R$ introduced in~\cite{hovey-christensen-relative hom alg}, and recently extended to any locally presented additive category in~\cite{krause-adjoints in homotopy cats}. Here if we take $G = \oplus G_i$ where the $G_i$ range through a set of isomorphism representatives for all finitely presented objects, then a complex $X$ is \emph{pure acyclic} if and only if $H_n[\Hom(G,X)]$ vanishes for all $n$. Similarly, isomorphisms in the pure derived category are those chain maps $X \xrightarrow{} Y$ which become homology isomorphisms after applying $\Hom(G,-) = \prod \Hom(G_i,-)$. So we are essentially doing homological algebra with respect to the generator $G$.

The most important categories we encounter in homological algebra are the Grothendieck categories, which recall are the abelian categories having exact direct limits and a generator $G$. A generator $G$ is equivalent to a generating set $\{G_i\}$ where $G = \oplus G_i$. This paper starts by showing that given any Grothendieck category $\cat{G}$ and a fixed choice of generator $G = \oplus G_i$, we can define the derived category $\class{D}(G)$. This category is obtained by inverting the \emph{$G$-homology isomorphisms}, which are the chain maps $X \rightarrow Y$ in $\cha{G}$ such that $\Hom_{\cat{G}}(G,X) \xrightarrow{} \Hom_{\cat{G}}(G,Y)$ is a homology isomorphism.  Said another way, this is the category obtained from $\cha{G}$ by forcing the \emph{$G$-acyclic complexes}, which are those complexes $X$ for which $\Hom_{\cat{G}}(G,X)$ is exact, to be 0.  To do this, we begin by showing that the generator $G = \oplus G_i$ determines a Quillen exact structure on $\cat{G}$, which as we prove in Appendix~\ref{appendix-proper classes}, is equivalent to a proper class of short exact sequences in the sense of~\cite{homology}. The short exact sequences here are precisely the usual short exact sequences which remain exact after applying $\Hom_{\cat{G}}(G,-)$. We call them \emph{$G$-exact sequences} and we denote this exact structure by $\cat{G}_G$. It becomes clear that we should define the $G$-derived category $\class{D}(G)$ to be $\class{D}(\cat{G}_G)$, the derived category with respect to the exact category $\cat{G}_G$, in the sense of~\cite{neeman-exact category} and~\cite{keller-derived cats}. 

But to get a deeper understanding of the $G$-derived category one would like to have a Quillen \emph{model} structure on $\cha{G}$ whose trivial objects are the $G$-acyclic complexes. In this case the associated homotopy category would coincide with $\class{D}(G)$. Such a model structure would first of all provide a convenient description of the morphism sets. But more importantly the theory of cofibrantly generated and monoidal model categories could be used to study $\class{D}(G)$. 
We in fact are able to build not just one, but two cofibrantly generated models on $\cha{G}$ whose trivial objects are the $G$-acyclic complexes. The first is a generalization of the usual projective model structure on $\ch$ while the second is a generalization of the usual injective model structure on $\ch$. See~\cite{hovey-model-categories} for details on these model structures.

To summarize, we use Hovey's correspondence between cotorsion pairs and abelian model structures to obtain the following result. 

\begin{them}[Models for $G$-derived categories]\label{them-A1}
Let $\cat{G}$ be any Grothendieck category with a generator $G = \oplus G_i$. 
\begin{enumerate}
\item There is a model structure on $\cha{G}$ which we call the \textbf{$\boldsymbol{G}$-projective model structure} whose trivial objects are the $G$-acyclic complexes. We call the associated homotopy category the \textbf{$\boldsymbol{G}$-derived category}, and denote it by $\class{D}(G)$. It is always a well generated triangulated category. 

\item If each $G_i$ is finitely presented then $\class{D}(G)$ is compactly generated. In this case we also have a dual model structure on $\cha{G}$ which we call the \textbf{$\boldsymbol{G}$-injective model structure}.

\item For given objects $A,B \in \class{G}$ we have
$\class{D}(G)(A,\Sigma^nB) = \GExt^n_{\cat{G}}(A,B)$ where $\GExt^n_{\cat{G}}(A,B)$ denotes the group of (equivalence classes of) $n$-fold $G$-exact sequences $B \rightarrowtail X_1 \rightarrow \cdots \rightarrow X_n  \twoheadrightarrow A$. $G$-projective resolutions (resp. $G$-injective coresolutions) provide cofibrant replacements in the $G$-projective model structure (resp. fibrant replacements in the $G$-injective model structure) and allow for computation of $\GExt^n_{\cat{G}}(A,B)$ in the usual manner.
\end{enumerate}
\end{them}

It should be pointed out that a careful reading of~\cite{hovey-christensen-relative hom alg} reveals that one can deduce the existence of the $G$-projective model structure above from their general Theorem~2.2. But while that theorem is more broad, our approach is different, and our results are very specific. To illuminate the analogy to the usual projective model structure on $\ch$, where $R$ is a ring, we give complete descriptions of the cofibrations, fibrations, and weak equivalences in the $G$-projective model structure. For example, the cofibrant objects are precisely the complexes $P$ for which each $P_n$ is a \emph{$G$-projective} (direct summand of a coproduct of the $G_i$) and such that any chain map $P \xrightarrow{} X$, with target $X$ $G$-acyclic, is null homotopic. The projective model structure is studied in Section~\ref{sec-the G-derived category of a locally finitely presented Grothendieck category}. In particular, see Theorem~\ref{them-projective model for G-derived} and Corollary~\ref{cor-projective model for G-derived} also Subsection~\ref{subsec-GExt computation}. 

On the other hand, constructing the dual $G$-injective model structure is far more technical than constructing the $G$-projective model. To do so we use the theory of purity from~\cite{adamek-rosicky}. In particular, the assumption that the generating set $\{G_i\}$ satisfies that each $G_i$ is finitely presented is equivalent to saying that $\cat{G}$ is a locally finitely presented Grothendieck category. This is precisely the setting in which a nice theory of purity holds. See~\cite{adamek-rosicky},\cite{crawley-boevey}, and Appendix~\ref{appendix-lambda pure}. We emphasize that this still includes the most important categories we encounter in homological algebra. For instance, the category of quasi-coherent sheaves over a quasi-compact and quasi-separated scheme is a locally finitely presented Grothendieck category by~\cite[Proposition~3.1]{garkusha-quasi-coherent}. The injective model structure is studied in Section~\ref{sec-injective model struc for G-derived}. In particular, see Subsection~\ref{subsec-inj model} with Theorem~\ref{them-injective cotorsion pair for G-derived} being the dual of Theorem~\ref{them-projective model for G-derived} and Corollary~\ref{cor-injective model for G-derived} being the dual of Corollary~\ref{cor-projective model for G-derived}. 

In Section~\ref{sec-recollement of Krause for G-derived} we go on to show that two recollement situations hold whenever we assume the $G_i$ are finitely presented. 
Note that the $G$-projective and $G$-injective model
structures are ``balanced'' in the sense that they share the same trivial objects. This is essentially the reason behind the following theorem. It is a $G$-derived version of a well known fact about $\class{D}(R)$.

\begin{them}[Verdier localization recollement for $G$-derived categories]\label{them-A}
Suppose $\cat{G}$ is a Grothendieck category and that $G = \oplus G_i$ is a generator with each $G_i$ finitely presented. Let $\class{D}(G)$ denote the $G$-derived category. Let $K(\cat{G})$ denote the homotopy category of all chain complexes and let $K_{G\text{-ac}}(\cat{G})$ denote the subcategory of all $G$-acyclic complexes. Then we have a recollement of triangulated categories:
\[
\xy
(-28,0)*+{K_{G\textnormal{-ac}}(\cat{G})};
(0,0)*+{K(\cat{G})};
(28,0)*+{\class{D}(G)};
{(-19,0) \ar (-8,0)};
{(-8,0) \ar@<0.5em> (-19,0)};
{(-8,0) \ar@<-0.5em> (-19,0)};
{(8,0) \ar (19,0)};
{(19,0) \ar@<0.5em> (8,0)};
{(19,0) \ar@<-0.5em> (8,0)};
\endxy
.\]
\end{them}

\begin{proof}
See Theorem~\ref{them-verdier recollement for G-derived} where the functors are described as well.
\end{proof}

The existence of the injective model structure will also lead us to the following Theorem, which is a $G$-version of Krause's result from~\cite{krause-stable derived cat of a Noetherian scheme}. Here we call an object \emph{$G$-injective} if it is injective with respect to the $G$-exact sequences already mentioned above.

\begin{them}[Krause's recollement for $G$-derived categories]\label{them-B}

Let $\cat{G}$ be a Grothen-dieck category and let $G = \oplus G_i$ be a generator with each $G_i$ finitely presented. Let $\class{D}(G)$ denote the $G$-derived category. Let $K_G(Inj)$ denote the homotopy category of all complexes of $G$-injectives. Let $K_{G\textnormal{-ac}}(Inj)$ denote the homotopy category of all $G$-acyclic complexes of $G$-injectives. Then there is a recollement
\[
\xy
(-30,0)*+{K_{G\textnormal{-ac}}(Inj)};
(0,0)*+{K_G(Inj)};
(26,0)*+{\class{D}(G)};
{(-19,0) \ar (-10,0)};
{(-10,0) \ar@<0.5em> (-19,0)};
{(-10,0) \ar@<-0.5em> (-19,0)};
{(10,0) \ar (19,0)};
{(19,0) \ar@<0.5em> (10,0)};
{(19,0) \ar@<-0.5em> (10,0)};
\endxy
.\]
\end{them}

\begin{proof}
See Theorem~\ref{them-krause recollement for G-derived}.
\end{proof}

The introduction continues in Section~\ref{sec-examples} where we list several applications or examples of the above Theorems.

\

\noindent \textbf{Acknowledgements:} The entire idea of the $G$-derived category was suggested by Mark Hovey. He pointed out that the author's flawed proof of Krause's recollement that appeared in an early version of~\cite{gillespie-recoll} might hold if one could replace the usual derived category with an alternate derived category obtained by ``killing the $G$-acyclic complexes''. The author thanks him for the idea and for several helpful suggestions while writing the paper. The author also wishes to thank the referee for useful comments and suggestions.

\section{Examples}\label{sec-examples}

As described in the Introduction, this paper shows that for a given set of generators $\{G_i\}$ in a Grothendieck category $\cat{G}$, we can do homological algebra by viewing everything ``through the eyes of $G$''. In particular, one should try to understand the proper class of $G$-exact sequences; those short exact sequences which remain exact after applying $\Hom_{\cat{G}}(G_i,-)$ for all the $G_i$. Whenever $G = \oplus G_i$ is projective, then this is just the usual class of short exact sequences and so $\class{D}(G)$ is the usual derived category $\class{D}(\cat{G})$. So the interesting thing is to explore what happens for other choices of $G$. We consider some examples here but there is much more room to explore this theme.

\subsection{Pure and $\boldsymbol{\lambda}$-pure derived categories.}\label{subsec-pure derived recollements in introduction}

In~\cite{hovey-christensen-relative hom alg}, Christensen and Hovey put a model structure on $\ch$ whose homotopy category was the \emph{pure derived category}, obtained by killing the pure acyclic complexes. More generally Krause shows in~\cite[Theorem~4.1]{krause-adjoints in homotopy cats} that the pure derived category $\class{D}_{\textnormal{pur}}(\cat{G})$ exists whenever $\cat{G}$ is a locally finitely presentable Grothendieck category. In this case he shows there is a recollement situation when passing from $K(\cat{G})$ to  $\class{D}_{\textnormal{pur}}(\cat{G})$.  This also follows from Theorem~\ref{them-A} by taking $G = \oplus G_i$ where the $G_i$ range through a set of isomorphism representatives for all finitely presented objects. (However, we note that Krause does not even assume that $\cat{G}$ is Grothendieck, merely additive.) But now we also have the following result as an immediate consequence of our above Theorem~\ref{them-B}.

\begin{them}\label{them-C}
Suppose that $\cat{G}$ is any locally finitely presentable Grothendieck category. Let $\class{D}_{\textnormal{pur}}(\cat{G})$ denote the pure derived category. Let $K(PInj)$ denote the homotopy category of all complexes of pure-injective objects in $\cat{G}$. Let $K_{\textnormal{p-ac}}(PInj)$ denote the homotopy category of all pure acyclic complexes of pure-injectives. Then there is a recollement
\[
\xy
(-30,0)*+{K_{\textnormal{p-ac}}(PInj)};
(0,0)*+{K(PInj)};
(28,0)*+{\class{D}_{\textnormal{pur}}(\cat{G})};
{(-19,0) \ar (-10,0)};
{(-10,0) \ar@<0.5em> (-19,0)};
{(-10,0) \ar@<-0.5em> (-19,0)};
{(10,0) \ar (19,0)};
{(19,0) \ar@<0.5em> (10,0)};
{(19,0) \ar@<-0.5em> (10,0)};
\endxy
.\]
\end{them}
 
Theorem~\ref{them-C} is interesting, assuming our category $\cat{G}$ admits chain complexes in $K_{\textnormal{p-ac}}(PInj)$ that are not contractible. We would like to have explicit examples of such complexes, or results indicating when such complexes do not exist.

We describe in Subsection~\ref{subsec-lambda pure derived cat} a generalization of the pure derived category to \emph{any} Grothendieck category by replacing the notion of pure with the notion of $\lambda$-pure where $\lambda$ is some large regular cardinal. We are only able to show that the projective model structure exists. But here a cofibrant replacement of an object $A \in \cat{G}$ is obtained by taking a $\lambda$-pure projective resolution of $A$ in the sense of~\cite{rosicky-prague talk}. It is worth noting that the existence of the $\lambda$-pure derived category doesn't appear to follow from results in~\cite{krause-adjoints in homotopy cats} because the $\lambda$-pure short exact sequences are not closed under filtered colimits, only $\lambda$-filtered colimits. For a similar reason, the $\lambda$-pure exact structure on $\cat{G}$ doesn't appear to be, in general, of \emph{Grothendieck type} in the sense of~\cite{stovicek-exact model cats}. We see in Subsection~\ref{subsec-lambda pure derived cat} that for any generator $G = \oplus G_i$, there is a regular cardinal $\lambda$ and a canonical functor $\class{D}_{\lambda\text{-pur}}(\cat{G})  \xrightarrow{} \class{D}(G)$ where $\class{D}_{\lambda\text{-pur}}(\cat{G})$ is the $\lambda$-pure derived category. This functor admits a left adjoint and provides a map of relative Ext groups $\lambda$-\textnormal{PExt}${}^n_{\cat{G}}(A,B) \rightarrow \GExt^n_{\cat{G}}(A,B)$ which is natural in $A,B \in \cat{G}$.

\subsection{Sheaves of modules on a ringed space.}\label{subsec-sheaves over a ringed space}

Let $\class{O}_X$ be a ringed space, that is, a sheaf of rings on a topological space $X$.
 The category $\class{O}_X\text{-Mod}$ of sheaves of $\class{O}_X$-modules is a Grothendieck category. Lets first recall the standard set of generators for $\class{O}_X$-Mod. For each open $U \subseteq X$, extend
$\class{O}_{|U}$ by 0 outside of $U$ to get a presheaf, which we denote
by $\class{O}_U$. Now sheafify to get an $\class{O}_X$-module,
which we will denote $j!(\class{O}_U)$. There are standard isomorphisms
$\Hom(j!(\class{O}_U) , G) \cong \Hom(\class{O}_U , G) \cong G(U)$. It follows at once that
the set $\{\, j!(\class{O}_U) \,\}$ forms a generating set since
the modules $j!(\class{O}_U)$ ``pick out points''. Hence the
direct sum $G = \bigoplus_{U \subseteq X} j!(\class{O}_U)$ is a
generator. The above isomorphisms also imply that the $G$-exact category is just $\class{O}_X\text{-Mod}$ together with the proper class of short \emph{presheaf exact} sequences of $\class{O}_X$-modules. That is, a $G$-exact sequence is an exact sequence $0 \rightarrow F \rightarrow G \rightarrow H \rightarrow 0$  of $\class{O}_X$-modules for which $0 \rightarrow F(U) \rightarrow G(U) \rightarrow H(U) \rightarrow 0$ is an exact sequence of $\class{O}(U)$-modules for each open $U \subseteq X$. The $G$-derived category of Theorem~\ref{them-A1} is thus the category of unbounded complexes of $\class{O}_X$-modules modulo the the \emph{presheaf acyclic} complexes.  Using, again, the above isomorphisms, it  follows immediately from~\cite[Exercise~II.1.11]{hartshorne} that each $j!(\class{O}_U)$ is finitely presented whenever the space $X$ is Noetherian. In particular, whenever $X = (X,\class{O}_X)$ is a Noetherian scheme then $\class{D}(G)$ is compactly generated. Also Theorems~\ref{them-A} and~\ref{them-B} apply in this case and the reader can interpret what they say. Just note that a $G$-injective $\class{O}_X$-module here translates to one that is injective with respect to the short presheaf exact sequences. By Proposition~\ref{prop-G-exact categories have enough injectives}, there are enough such $G$-injectives in the sense that we can find for any $\class{O}_X$-module $F$ a short presheaf exact sequence $0 \rightarrow F \rightarrow I \rightarrow I/F \rightarrow 0$ where $I$ is $G$-injective.

\subsection{Quasi-coherent sheaves over the projective line $\textbf{P}^1(k)$.}\label{subsec-quasi coherent sheaves on P1(k)}

Let $k$ be a commutative ring with identity. Here we consider the category of quasi-coherent sheaves over the projective line $\textbf{P}^1(k)$. However, we use the quiver description of this category from~\cite{enochs-quasi-coherent}, \cite{enochs-estrada-flat cotorsion sheaves. applications.}, \cite{enochs-quasi-coherent-projectiveline} and~\cite{enochs-estrada-locally projective monoidal model}. From this point of view, we consider the representation $$R \equiv k[x] \hookrightarrow k[x,x^{-1}] \hookleftarrow k[x^{-1}]$$ of the quiver $Q \equiv \bullet \rightarrow \bullet \leftarrow \bullet$. Then $R$ corresponds to the structure sheaf on $\textbf{P}^1(k)$. A quasi-coherent sheaf of modules over $\textbf{P}^1(k)$ may be thought of as a representation $$A \equiv M \xrightarrow{f} L \xleftarrow{g} N$$ with $M$ a $k[x]$-module, $L$ a $k[x,x^{-1}]$-module, $N$ a $k[x^{-1}]$-module,  $f$ a $k[x]$-linear map, and $g$ a $k[x^{-1}]$-linear map; all satisfying that the localization maps $S^{-1}f : S^{-1}M \xrightarrow{} S^{-1}L \cong L$ and $T^{-1}g : T^{-1}N \xrightarrow{} T^{-1}L \cong L$ are $k[x,x^{-1}]$-isomorphisms, where $S = \{1,x,x^2,\cdots\}$ and $T = \{1,x^{-1},x^{-2},\cdots\}$. We call such an $A$ a \emph{quasi-coherent $R$-module}. A morphism is the obvious triple of linear maps providing commutative squares. Denote by $\qcor$ the category of all quasi-coherent $R$-modules. Then $\qcor$ is equivalent to the category of quasi-coherent sheaves on $\textbf{P}^1(k)$ and so it is a Grothendieck category. There is a set of generators corresponding to the line bundles of degree $n$ over $\textbf{P}^1(k)$. They are the quasi-coherent $R$-modules  $$R(n) \equiv k[x] \hookrightarrow k[x,x^{-1}] \xleftarrow{x^n} k[x^{-1}]\, , \ \ \ n \in \Z$$ where the map on the right is multiplication by $x^n$. Tensor products, direct limits, and finite limits are all taken componentwise. In particular, a short exact sequence in $\qcor$ is one having all three involved short sequences exact. We refer the reader to~\cite{enochs-quasi-coherent}, \cite{enochs-estrada-flat cotorsion sheaves. applications.}, \cite{enochs-quasi-coherent-projectiveline} and~\cite{enochs-estrada-locally projective monoidal model} for more detail on all of the above.

Now given any $A \in \qcor$, by regarding it as a diagram $M \xrightarrow{f} L \xleftarrow{g} N$ of just abelian groups, we may take the pullback $M \times_L N$. Denote this abelian group by $PA$. Also, given an integer $n$, denote by $A(n)$ the \emph{twisted sheaf} $R(n) \tensor_R A$. Note that there is an obvious isomorphism
$A(n) \equiv M \xrightarrow{f} L \xleftarrow{x^n \cdot g} N$. Each $R(n)$ is flat and in particular if $0 \xrightarrow{} A \xrightarrow{} B
\xrightarrow{} C \xrightarrow{} 0$ is a short exact sequence in $\qcor$, then so is $0 \xrightarrow{} A(n) \xrightarrow{} B(n)
\xrightarrow{} C(n) \xrightarrow{} 0$. Consequently we have that $0 \xrightarrow{} PA(n) \xrightarrow{} PB(n)
\xrightarrow{} PC(n)$ is exact. If each $PB(n) \xrightarrow{} PC(n)$ is also onto, then lets refer to $0 \xrightarrow{} A \xrightarrow{} B
\xrightarrow{} C \xrightarrow{} 0$ as a \emph{twisted fibre exact sequence}.

From~\cite{enochs-quasi-coherent-projectiveline} we have that $\{R(n)\}$ is a set of (flat) generators for $\qcor$. Setting $G = \oplus_{n \in \Z} R(n)$, one can show that the $G$-exact sequences are precisely the twisted fibre exact sequences. Indeed for each $n$ one can check directly that the elements of $\Hom_{\qcor}(R(n),A)$ are in one to one correspondence with the elements of the pullback $PA(-n)$. That is, we have natural isomorphisms of abelian groups $\Hom_{\qcor}(R(n),A) \cong PA(-n)$. This isomorphism also can be used to show that each $R(n)$ is finitely presented: For a direct limit $\varinjlim A_i$, using that pullbacks and tensor products commute with direct limits we see $$\Hom_{\qcor}(R(n),\varinjlim A_i) \cong P[(\varinjlim A_i)(-n)] \cong P[R(-n) \tensor_R \varinjlim A_i] \cong$$ $$P[\varinjlim(R(-n) \tensor_R A_i)] \cong \varinjlim P[R(-n) \tensor_R A_i]  \cong \varinjlim PA_i(-n)$$  $$\cong \varinjlim \Hom_{\qcor}(R(n),A_i).$$ So Theorems~\ref{them-A1},~\ref{them-A}, and~\ref{them-B} apply. Moreover, our characterization of the cofibrant and trivially cofibrant objects provided by Theorem~\ref{them-projective model for G-derived} allows one to easily check that the model structure is \emph{monoidal} so that the tensor product descends to a well-behaved tensor product on the $G$-derived category. To do this, apply Hovey's~\cite[Theorem~7.2]{hovey} and the method of~\cite[Theorem~5.1]{gillespie-quasi-coherent}; it all boils down to the fact that $R(m) \tensor_R R(n) \cong R(m+n)$ which was shown from the quiver perspective in~\cite[Proposition~3.3]{enochs-quasi-coherent-projectiveline}.

\subsection{Other examples concerning modules over a ring.}\label{subsec-examples concerning modules over a ring}

Let $R$ be a ring with 1, and let $\cat{G} = R\text{-Mod}$ be the category of (left) $R$-modules.
Note that if $\class{S}$ is \emph{any} set of $R$-modules, then $\class{S} \cup \{R\}$ is a generating set for $\rmod$. So Theorem~\ref{them-A1} gives us a model structure killing the exact complexes which remain exact after applying $\Hom_R(S,-)$ for all $S \in \class{S}$. Of course Theorems~\ref{them-A} and~\ref{them-B} also hold if all the $S$ are finitely presented modules. Moreover, whenever $\class{S} \subseteq \class{T}$, then in a way analogous to Corollary~\ref{cor-lambda derived cat} we have a canonical functor $\class{D}(\class{T}) \xrightarrow{} \class{D}(\class{S})$ with a left adjoint. The functor provides a mapping of relative $\Ext$ groups.
We give two interesting examples below.

\subsubsection{The clean derived category}\label{example-clean derived cat}
For non-coherent rings we have the following variant of the pure derived category. An $R$-module is said to be of \emph{type $FP_{\infty}$} if it has a projective resolution consisting of finitely generated free modules. The category of all type $FP_{\infty}$ modules has a small skeleton. So we can take $\class{S}$ to be a set of isomorphism representatives. Then with $G = \oplus_{S \in \class{S}} S$ we get that the $G$-exact category $\cat{G}_G$ is exactly the category of $R$-modules along with the proper class of all \emph{clean exact sequences} in the sense of~\cite{bravo-gillespie-hovey}. The injectives in $\cat{G}_G$ ought to be called \emph{clean injective} modules. The projectives in $\cat{G}_G$ are precisely direct summands of direct sums of modules of type $FP_{\infty}$. Since all modules of type $FP_{\infty}$ are finitely presented, Theorems~\ref{them-A1},~\ref{them-A} and~\ref{them-B} apply giving recollements involving the \emph{clean derived category}. We see a canonical functor from the pure derived category to the clean derived category. However, we point out that for coherent rings, a module is finitely presented if and only if it is of type $FP_{\infty}$. So this example only differs from the pure derived category for non-coherent rings.

It seems likely that the clean derived category will generalize to some other locally finitely presented Grothendieck categories. By~\cite[Corollary~1.6]{bieri} we have that for modules over a ring, $F$ is of type $FP_{\infty}$ if and only if $\Ext^n_R(F,-)$ preserves direct limits for all $n \geq 0$. So in the more general setting, even without enough projective objects, one could define an object $F \in \cat{G}$ to be of type $FP_{\infty}$ if $\Ext^n_{\cat{G}}(F,-)$ preserves direct limits for all $n \geq 0$. However, one needs to be sure that the objects of type $FP_{\infty}$ form a generating set for $\cat{G}$!

\subsubsection{Inj-acyclic complexes}\label{example-clean derived cat}
Suppose $R$ is (left) Noetherian. Recall that every injective (left) $R$-module is a direct sum of indecomposable injective modules and there is a set $\class{S}$ of (isomorphism representatives) of all indecomposable injectives. (See~\cite[Theorem~3.48]{lam}.) So taking $G$ to be the direct sum of $R$ and all the indecomposable injectives, it is easy to see that a short exact sequence is $G$-exact if and only if it remains exact after applying $\Hom_R(I,-)$ where $I$ is any injective $R$-module. So these are a proper class of short exact sequences and the injective modules are projective objects with respect to these. More generally, by part~(4) of Corollary~\ref{cor-G-projective generator}, the $G$-projectives are precisely the direct summands of direct sums of modules in $\class{S} \cup \{R\}$. By Theorem~\ref{them-A1}, we get a model structure for an associated derived category obtained by killing all the exact ``Inj-acyclic'' complexes.

\section{The G-exact category $\cat{G}_{G}$}\label{sec-The G-exact category of a locally finitely presented Grothendieck category}

Throughout this section $\cat{G}$ will always denote a Grothendieck category with a chosen (fixed) set of generators $\{G_i\}_{i \in I}$. Furthermore, $G$ will always denote their direct sum $G = \oplus_{i \in I} G_i$. So $G$ itself is a generator for $\cat{G}$. The goal of this section is to give a detailed construction of an exact category, in the sense of Quillen~\cite{quillen-algebraic K-theory} and~\cite{buhler-exact categories}, which we will call the $G$-exact category of $\cat{G}$. Being abelian, an exact structure on $\cat{G}$ is, as shown in Appendix~\ref{appendix-proper classes}, nothing more than a proper class of short exact sequences in the sense of~\cite{homology}. In this case, the proper class is the class of all $G$-exact sequences. That is, the short exact sequences $0 \xrightarrow{} A \xrightarrow{} B \xrightarrow{} C \xrightarrow{} 0$ which remain exact after applying $\Hom_{\cat{G}}(G,-)$. We denote this exact category by $\cat{G}_G$, and see that $G$ is a projective generator for $\cat{G}_G$.

\subsection{$\boldsymbol{G}$-exact sequences and $\boldsymbol{G}$-projectives.}
Recall that an object $G$ in an abelian category $\cat{A}$ is a \emph{generator} if $\Hom_{\cat{A}}(G,-)$ is faithful. Since $\cat{A}$ is abelian this is equivalent to saying that if $f : A \xrightarrow{} B$ is nonzero, then there exists a map $s : G \xrightarrow{} A$ such that $fs \neq 0$. We have the following basic fact.

\begin{lemma}\label{lemma-generators reflect exactness}
Let $G$ be a generator for any abelian category $\cat{A}$ and let $X$ be a chain complex in $\cha{A}$. If the complex of abelian groups $\Hom_{\cat{A}}(G,X)$ is exact, then $X$ itself must be exact.

\end{lemma}

\begin{proof}
We just need to show that $d_{n+1} : X_{n+1} \xrightarrow{} Z_nX$ is an epimorphism, that is, right cancelable. Since $\cat{A}$ is abelian we just need to show that for a map $f : Z_nX \xrightarrow{} Y$ we have $fd_{n+1} = 0$ implies $f = 0$. By way of contradiction, say $fd_{n+1} = 0$ but $f \neq 0$. Then because $G$ is a generator we get a map $s : G \xrightarrow{} Z_nX$ such that $fs \neq 0$. But notice $s$ determines a map in the domain of $(d_n)_* : \Hom_{\cat{A}}(G, X_n) \xrightarrow{} \Hom_{\cat{A}}(G, X_{n-1})$ for which $(d_n)_*(s) = 0$. So by hypothesis we have $s \in \ker{(d_n)_*} = \im{(d_{n+1})_*}$ which ensures a map $t : G \xrightarrow{} X_{n+1}$ such that $s = d_{n+1}t$. Now $fd_{n+1} = 0$ implies $fd_{n+1}t = 0$ implies $fs = 0$, which is the contradiction.
\end{proof}

Now let $R = \Hom_{\cat{G}}(G,G)$ be the endomorphism ring of $G$ and let $\text{Mod-}R$ be the category of right $R$-modules. By the Gabriel-Popescu Theorem, the functor $\Hom_{\cat{G}}(G,-) : \cat{G} \xrightarrow{} \text{Mod-}R$ is fully faithful and has an exact left adjoint $T$. Therefore $\cat{G}$ is equivalent to the full subcategory $\cat{S} = \im{[\Hom_{\cat{G}}(G,-)]}$ of $\text{Mod-}R$. Since the property of being a Grothendieck category is stable under equivalence of categories we know that $S$ is Grothendieck. However $S$ is not an abelian subcategory of $\text{Mod-}R$. In particular, if $0 \xrightarrow{} A \xrightarrow{f} B \xrightarrow{g} C \xrightarrow{} 0$ is a short exact sequence in $\cat{G}$, then of course $0 \xrightarrow{} \Hom_{\cat{G}}(G,A) \xrightarrow{f_*} \Hom_{\cat{G}}(G,B) \xrightarrow{g_*} \Hom_{\cat{G}}(G,C)$ is generally only a left exact sequence in $\text{Mod-}R$. But this IS a short exact sequence in the abelian category $\cat{S}$. Indeed lets show directly that $g_*$ is right cancelable in $\cat{S}$, making it an epimorphism in $\cat{S}$: Given any morphism $t : \Hom_{\cat{G}}(G,C) \xrightarrow{} S$ in $\cat{S}$, we wish to show $0 = t g_*$ implies $0=t$. But $\Hom_{\cat{G}}(G,-)$ is full and so $t$ must take the form $\Hom_{\cat{G}}(G,C) \xrightarrow{h_*} \Hom_{\cat{G}}(G,D)$ for some $h : C \xrightarrow{} D$ in $\cat{G}$. So we have $0 = tg_* = h_*g_* = (hg)_*$. Since $\Hom_{\cat{G}}(G,-)$ is faithful we have $hg = 0$. But $g$ is right cancelable, so $h=0$ and this implies $h_* = t=0$.

\begin{definition}
We call a pair of composeable maps $A \xrightarrow{f} B \xrightarrow{g} C$ in $\cat{G}$ a \textbf{$\boldsymbol{G}$-exact sequence}  if $0 \xrightarrow{} \Hom_{\cat{G}}(G,A) \xrightarrow{f_*} \Hom_{\cat{G}}(G,B) \xrightarrow{g_*} \Hom_{\cat{G}}(G,C) \xrightarrow{} 0$ is a short exact sequence in the category of abelian groups (so also in $\text{Mod-}R$). We often  denote a $G$-exact sequence by $A \rightarrowtail B \twoheadrightarrow C$, and call $A \rightarrowtail B$ a \textbf{$\boldsymbol{G}$-monomorphism} and $B \twoheadrightarrow C$ a \textbf{$\boldsymbol{G}$-epimorphism}. We will also call a subobject $P \subseteq A$ a \textbf{$\boldsymbol{G}$-subobject} if the inclusion map is a $G$-monomorphism, and denote this $P \subseteq_G A$.

\end{definition}

We list some basic properties of $G$-exact sequences.

\begin{proposition}\label{prop-G-exact sequences-properties}
We have the following properties of $G$-exact sequences.
\begin{enumerate}
\item Any $G$-exact sequence is an exact sequence in $\cat{G}$.
\item The class of all $G$-exact sequences is closed under isomorphisms and contains all split exact sequence.
\item A pushout of a $G$-monomorphism is again a $G$-monomorphism. In fact, $\Hom_{\cat{G}}(G,-)$ takes pushouts of $G$-monomorphisms to pushouts in $\text{Mod-}R$. We also have that pullbacks of $G$-epimorphisms are again $G$-epimorphisms. Moreover, $\Hom_{\cat{G}}(G,-)$ takes all pullbacks in $\cat{G}$ to pullbacks in $\text{Mod-}R$ since it is a right adjoint.
\item $G$-monomorphisms are closed under composition and $G$-epimorphisms are closed under composition.
\end{enumerate}
\end{proposition}

\begin{proof}
For (1), note that in the definition of $G$-exact sequence we have $0 = g_*f_* = (gf)_*$. So $\Hom_{\cat{G}}(G,-)$ faithful implies $0 = gf$. So we can view $A \xrightarrow{f} B \xrightarrow{g} C$ as a chain complex in $\cat{G}$, and so (1) follows from Lemma~\ref{lemma-generators reflect exactness}.

(2) is clear.

For (3), we first show that a pullback of a $G$-epimorphism is a $G$-epimorphism. Let $0 \xrightarrow{} A \xrightarrow{f} B \xrightarrow{g} C \xrightarrow{} 0$ be a given $G$-exact sequence. Taking a pullback $B \xrightarrow{g} C \xleftarrow{} X$ leads to a diagram of short exact sequences.
$$\begin{CD}
0       @>>>     A    @>f'>> P @>g'>> X @>>> 0  \\
@. @|             @VVV   @VVV @.            \\
0       @>>>  A @>f>> B    @>g>> C @>>> 0\\
\end{CD}$$
Applying $\Hom_{\cat{G}}(G,-)$ to this diagram gives us a commutative diagram with the bottom row exact
$$\begin{CD}
0       @>>>     \Hom_{\cat{G}}(G,A)    @>f'_*>> \Hom_{\cat{G}}(G,P) @>g'_*>> \Hom_{\cat{G}}(G,X) \\
@. @|             @VVV   @VVV @.            \\
0       @>>>  \Hom_{\cat{G}}(G,A) @>f_*>> \Hom_{\cat{G}}(G,B)    @>g_*>> \Hom_{\cat{G}}(G,C) @>>> 0\\
\end{CD}$$
But the functor $\Hom_{\cat{G}}(G,-) : \cat{G} \xrightarrow{} \text{Mod-}R$ is a right adjoint and so it preserves limits, so in particular it preserves pullbacks. Therefore the right square is a pullback in $\text{Mod-}R$. So since $g_*$ is an epimorphism we get that $g'_*$ must also be an epimorphism. This proves $0 \xrightarrow{} A \xrightarrow{f'} P \xrightarrow{g'} X \xrightarrow{} 0$ is a $G$-exact sequence.

Next, we wish to show that a pushout of a $G$-monomorphism is a $G$-monomorphism. So consider a $G$-exact sequence $0 \xrightarrow{} A \xrightarrow{f} B \xrightarrow{g} C \xrightarrow{} 0$. Taking a pushout of $X \xleftarrow{} A \xrightarrow{f} B$ leads to a diagram of short exact sequences.
$$\begin{CD}
0       @>>>     A    @>f>> B @>g>> C @>>> 0  \\
@.            @VVV   @VVV @|     @.   \\
0       @>>>  X @>f'>> P    @>g'>> C @>>> 0\\
\end{CD}$$ We only need to show that $g'_*$ is an epimorphism. Since $\Hom_{\cat{G}}(G,-)$ is not a left adjoint we can't expect it to preserve all pushouts. However, note that since $\Hom_{\cat{G}}(G,-) : \cat{G} \xrightarrow{} \cat{S}$ is an equivalence it takes pushouts in $\cat{G}$ to pushouts in the abelian category $\cat{S}$. This implies that we get the $\cat{S}$-diagram below with $\cat{S}$-exact rows and with the left square being a pushout in $\cat{S}$.
$$\begin{CD}
0       @>>>     \Hom_{\cat{G}}(G,A)    @>f_*>> \Hom_{\cat{G}}(G,B) @>g_*>> \Hom_{\cat{G}}(G,C) \\
@. @VVV             @VVV   @|             \\
0       @>>>  \Hom_{\cat{G}}(G,X) @>f'_*>> \Hom_{\cat{G}}(G,P)    @>g'_*>> \Hom_{\cat{G}}(G,C) \\
\end{CD}$$
But by hypothesis, $g_*$ is an epimorphism in $\text{Mod-}R$, and so we see immediately that $g'_*$ is also an epimorphism in $\text{Mod-}R$. This shows that  $X \xrightarrow{f'} P \xrightarrow{g'} C$ is a $G$-exact sequence. In fact, since the rows of the diagram above are exact in $\text{Mod-}R$, it follows that the left hand square is actually the pushout in $\text{Mod-}R$. So the functor $\Hom_{\cat{G}}(G,-) : \cat{G} \xrightarrow{} \text{Mod-}R$ preserves pushouts of $G$-monomorphisms.

For (4), we first show that $G$-epimorphisms are closed under composition. Say $B \xrightarrow{g} C$ and $C \xrightarrow{h} D$ are each $G$-epimorphisms. Since each is an epimorphism, so is the composition $hg$. Then $0 \xrightarrow{} \ker{hg} \xrightarrow{} B \xrightarrow{hg} D \xrightarrow{} 0$ must be a $G$-exact sequence since $(hg)_* = h_*g_*$ is an epimorphism.

Finally, we wish to show that $G$-monomorphisms are closed under composition. So let  $A \xrightarrow{i} B$ and $B \xrightarrow{j} C$ each be $G$-monomorphisms. Taking the pushout of $B/A \xleftarrow{} B \xrightarrow{j} C$ leads to a diagram of short exact sequences.
$$\begin{CD}
@. 0 @. 0 \\
@. @VVV @VVV \\
@. A @= A \\
@. @ViVV @VjiVV \\
0       @>>>    B    @>j>> C @>\pi>> C/B @>>> 0  \\
@.            @VVV   @VgVV @|     @.   \\
0       @>>>  B/A @>j'>> P    @>\pi'>> C/B @>>> 0\\
@. @VVV @VVV \\
@.   0  @.   0 \\
\end{CD}$$
Since the row $0 \xrightarrow{} B \xrightarrow{j} C \xrightarrow{\pi} C/B \xrightarrow{} 0$ is $G$-exact, we have by what was proved already that the pushout row $0 \xrightarrow{} B/A \xrightarrow{j'} P \xrightarrow{\pi'} C/B \xrightarrow{} 0$ must also be $G$-exact. So applying $\Hom_{\cat{G}}(G,-)$ yields a commutative diagram with exact rows and columns.
$$\begin{CD}
@. 0 @. 0 \\
@. @VVV @VVV \\
@. \Hom_{\cat{G}}(G,A) @= \Hom_{\cat{G}}(G,A) \\
@. @Vi_*VV @V(ji)_*VV \\
0       @>>>    \Hom_{\cat{G}}(G,B)    @>j_*>> \Hom_{\cat{G}}(G,C) @>\pi_*>> \Hom_{\cat{G}}(G,C/B) @>>> 0  \\
@.            @VVV   @Vg_*VV @|     @.   \\
0       @>>>  \Hom_{\cat{G}}(G,B/A) @>j'_*>> \Hom_{\cat{G}}(G,P)    @>\pi'_*>> \Hom_{\cat{G}}(G,C/B) @>>> 0\\
@. @VVV  \\
@.   0  @.  \\
\end{CD}$$
We are trying to show that $g_*$ is an epimorphism in $\text{Mod-}R$, and now the snake lemma shows that it is.
\end{proof}

We show in Appendix~\ref{appendix-proper classes} that when working in abelian categories, Quillen's notion of an \emph{exact category} from~\cite{quillen-algebraic K-theory} coincides with the notion of a \emph{proper class of short exact sequences} from~\cite[Chapter~XII.4]{homology}.

\begin{corollary}\label{cor-G-exacts form an exact category}
Let $\cat{G}$ be a Grothendieck category with generator $G$. Let $\class{E}$ denote the class of all $G$-exact sequences. Then $(\cat{G},\class{E})$ is an exact category. Equivalently, $\class{E}$ is a proper class of short exact sequences. We will let $\cat{G}_G = (\cat{G},\class{E})$ denote this exact category and we will call it the \textbf{G-exact category} of $\cat{G}$. The functor $\Hom_{\cat{G}}(G,-) : \cat{G}_G \xrightarrow{} \text{Mod-}R$ is exact.
\end{corollary}

\begin{proof}
The four properties of Proposition~\ref{prop-G-exact sequences-properties} are the axioms of an exact category in~\cite{buhler-exact categories}. It is clear from definitions that the functor $\Hom_{\cat{G}}(G,-) : \cat{G}_G \xrightarrow{} \text{Mod-}R$ is exact. We refer the reader to Appendix~\ref{appendix-proper classes} for the equivalence with proper classes.
\end{proof}

The generator $G = \oplus_{i \in I} G_i$ is not just a generator for $\cat{G}$. It is easy to see that it is also a generator for $\cat{G}_G$, but we first explain what we mean by this.

In~\cite{hovey}, Hovey worked with abelian categories along with a proper class of short exact sequences in the sense of~\cite[Chapter~XII.4]{homology}. There he defined an object $U$ to be a generator for a proper class $\class{P}$ if for all maps $f$, $\Hom_{\cat{G}}(U,f)$ surjective implies $f$ is a $\class{P}$-epimorphism. Also here, a set $\{U_i\}$ generates $\class{P}$ if $U = \oplus U_i$ is a generator for $\class{P}$. On the other hand, in~\cite{saorin-stovicek} and~\cite{stovicek-exact model cats}, the authors work with exact categories and define a set $\{U_i\}$ to be generating if for any object $A$, there is an admissible epimorphism $\pi : U \twoheadrightarrow A$ where $U$ is some set-indexed direct sum of objects from  $\{U_i\}$. The following corollary shows that $G$ is a generator for $\cat{G}_G$ in both senses. We therefore can feel free to reference the above authors' results.

\begin{corollary}\label{cor-G-projective generator}
$G = \oplus_{i \in I} G_i$ is a projective generator for the $G$-exact category $\cat{G}_G$. In particular, the following hold:
\begin{enumerate}
\item By definition, an object $P$ is projective in $\cat{G}_G$ if the functor $\Hom_{\cat{G}}(P,-)$ takes $G$-exact sequences to short exact sequences. We will call such an object \textbf{$\boldsymbol{G}$-projective}. Notice that the construction of the $G$-exact category immediately forces $G$ and each $G_i$ to be $G$-projective.
\item $G$ is a generator for $\cat{G}_G$. That is, if $\Hom_{\cat{G}}(G,A) \xrightarrow{f_*} \Hom_{\cat{G}}(G,B)$ is surjective, then $f$ is a $G$-epimorphism.

\noindent Equivalently, $\{G_i\}$ is a set of generators for $\cat{G}_G$. That is, if $\Hom_{\cat{G}}(G_i,A) \xrightarrow{f_*} \Hom_{\cat{G}}(G_i,B)$ is surjective for all $G_i$, then $f$ is a $G$-epimorphism.

\item $\cat{G}_G$ has enough projectives. In particular, for each $A \in \cat{G}$, we can find a $G$-epimorphism $\oplus_{i \in I} G \twoheadrightarrow A$. Equivalently, we can find a $G$-epimorphism $X \twoheadrightarrow A$ where $X$ is a direct sum of copies of some of the $G_i$.

\item An object $P$ is $G$-projective if and only if it is a direct summand of a direct sum of copies of some of the $G_i$.
\end{enumerate}
\end{corollary}

\begin{proof}
For (2), let $f : A \xrightarrow{} B$ be such that $\Hom_{\cat{G}}(G,A) \xrightarrow{f_*} \Hom_{\cat{G}}(G,B)$ is surjective. Since $G$ is a generator for $\cat{G}$ this implies $f$ is an epimorphism and so there is a short exact sequence $0 \xrightarrow{} K \xrightarrow{} A \xrightarrow{f} B \xrightarrow{} 0$. By definition this sequence is $G$-exact, so we are done. In terms of the generating set $\{G_i\}$, just note that $\Hom_{\cat{G}}(G,A) \xrightarrow{f_*} \Hom_{\cat{G}}(G,B)$ is surjective iff $\Hom_{\cat{G}}(G_i,A) \xrightarrow{f_*} \Hom_{\cat{G}}(G_i,B)$ is surjective for all $G_i$.

For (3), in the usual way, take $I = \Hom_{\cat{G}}(G,A)$, and define $\oplus_{t \in I} G \twoheadrightarrow A$ in component $(t : G \xrightarrow{} A) \in I$ to be $t$ itself. It is immediate that this is a $G$-epimorphism. $\oplus_{t \in I} G$ is indeed a $G$-projective object, since in any exact category, direct sums of projectives are again projectives by~\cite[Corollary~11.7]{buhler-exact categories}. For (4), we see that the $G$-epimorphism $\oplus_{t \in I} G \twoheadrightarrow P$ splits if and only if $P$ is $G$-projective by~\cite[Corollary~11.6]{buhler-exact categories}.
\end{proof}

\subsection{$\boldsymbol{G}$-subobjects.}

Here we go on to list more properties of $G$-monomorphisms, but we state them in terms of $G$-subobjects. This is the form in which we will use them later. Note that they are analogous to properties of pure submodules. Recall that we write $P \subseteq_G A$ to mean that $P$ is a $G$-subobject of $A$, that is, $\Hom_{\cat{G}}(G,A) \xrightarrow{} \Hom_{\cat{G}}(G,A/P)$ is surjective.

\begin{proposition}\label{prop-G-purity properties}
Consider subobject $A \subseteq B \subseteq C$ in $\cat{G}$.
\begin{enumerate}
\item If $A \subseteq_G B$ and $B \subseteq_G C$ then $A \subseteq_G C$.
\item If $A \subseteq_G C$ then $A \subseteq_G B$.
\item If $A \subseteq_G C$ and $B/A \subseteq_G C/A$ then $B \subseteq_G C$.
\end{enumerate}
\end{proposition}

\begin{proof}
(1) has already appeared as part~(4) of Proposition~\ref{prop-G-exact sequences-properties}. (2) follows from general facts about admissible monomorphisms in (weakly idempotent complete) exact categories. See~\cite[Prop.~7.6 or Prop.~2.16]{buhler-exact categories}.

For (3), all we need to check is that the map $\Hom_{\cat{G}}(G,C) \xrightarrow{} \Hom_{\cat{G}}(G,C/B)$ is an epimorphism. But this is just the composite $$\Hom_{\cat{G}}(G,C) \xrightarrow{} \Hom_{\cat{G}}(G,C/A) \xrightarrow{} \Hom_{\cat{G}}(G,(C/A)/(B/A)) \cong \Hom_{\cat{G}}(G,C/B),$$ and these are epimorphisms by hypothesis.
\end{proof}

\section{The $G$-derived category}\label{sec-the G-derived category of a locally finitely presented Grothendieck category}

Again let $\cat{G}$ be a Grothendieck category and let $G = \oplus G_i$ where $\{G_i\}$ is a set of generators. In this section we construct the derived category $\class{D}(G)$. It is the derived category of the $G$-exact category $\cat{G}_G$ and we obtain it by putting a suitable model structure on $\cha{G}$. Following the general definition of an exact chain complex from~\cite[Definition~10.1]{buhler-exact categories}, the exact complexes in $\cat{G}_G$ are the $G$-acyclic complexes. That is, those chain complexes $X$ for which $\Hom_{\cat{G}}(G,X)$ is exact. So we wish to ``kill'' these complexes by making them the trivial objects of an exact model structure.

\subsection{The category $\boldsymbol{\text{Ch}(\cat{G})_G}$.}

Our convention when working with chain complexes is that the differential lowers degree, so $\cdots
\xrightarrow{} X_{n+1} \xrightarrow{d_{n+1}} X_{n} \xrightarrow{d_n}
X_{n-1} \xrightarrow{} \cdots$ is a chain complex. Given $X \in \cha{G}$, the
\emph{$n$th suspension of $X$}, denoted $\Sigma^n X$, is the complex given by
$(\Sigma^n X)_{k} = X_{k-n}$ and $(d_{\Sigma^n X})_{k} = (-1)^nd_{k-n}$.
Given two chain complexes $X$ and $Y$ we define $\homcomplex(X,Y)$ to
be the complex of abelian groups $ \cdots \xrightarrow{} \prod_{k \in
\Z} \Hom(X_{k},Y_{k+n}) \xrightarrow{\delta_{n}} \prod_{k \in \Z}
\Hom(X_{k},Y_{k+n-1}) \xrightarrow{} \cdots$, where $(\delta_{n}f)_{k}
= d_{k+n}f_{k} - (-1)^n f_{k-1}d_{k}$.
This gives a functor
$\homcomplex(X,-) \mathcolon \cha{A} \xrightarrow{} \textnormal{Ch}(\Z)$. Note that this functor takes exact sequences to left exact sequences,
and it is exact if each $X_{n}$ is projective. Similarly the contravariant functor $\homcomplex(-,Y)$ sends exact sequences to left exact sequences and is exact if each $Y_{n}$ is injective. It is an exercise to check that the homology satisfies $H_n[Hom(X,Y)] = \cha{G}(X,\Sigma^{-n} Y)/\sim$ where $\sim$ is the usual relation of chain homotopic maps.

For a given $A \in \cat{G}$, we denote the \emph{$n$-disk on $A$} by $D^n(A)$. This is the complex consisting only of $A \xrightarrow{1_A} A$ concentrated in degrees $n$ and $n-1$. We denote the \emph{$n$-sphere on $A$} by $S^n(A)$, and this is the complex consisting of $A$ in degree $n$ and 0 elsewhere.

Recall that $\cat{G}_G$ is the same category as $\cat{G}$, with the same morphisms, but with an exact structure coming from the proper class of $G$-exact sequences. In the same way, we let $\cha{G}_G$ denote the category of all chain complexes, with the usual chain maps, but considered as an exact category where the short exact sequences are $G$-exact in each degree. We will call these \emph{degreewise $G$-exact sequences}. It is indeed a general fact that for any exact category $\cat{A} = (\cat{A},\class{E})$, the category $\cha{A}$ becomes an exact category when considered along with the short exact sequences which degreewise lie in $\class{E}$. So one might argue that the proper notation in our case is $\textnormal{Ch}(\cat{G}_G)$, rather than $\cha{G}_G$. However, we have the following lemma.

\begin{lemma}\label{lemma-generators in Ch(G)}
Consider the standard generating set $\{D^n(G_i)\}$ in $\cha{G}$ and let $G = \oplus D^n(G_i)$ be the direct sum, taken over all $n \in \Z$ and $i \in I$. Then the $G$-exact category $\cha{G}_G$ of Corollary~\ref{cor-G-exacts form an exact category} coincides with $\textnormal{Ch}(\cat{G}_G)$. That is, the proper class of $G$-exact sequences in $\cha{G}$ (here $G = \oplus D^n(G_i)$) coincides with the class of all short exact sequences which degreewise are $G$-exact sequences (here $G = \oplus G_i$) in $\cat{G}$.
\end{lemma}

\begin{proof}
Consider a short sequence $X \rightarrowtail Y \twoheadrightarrow Z$ of complexes. Then it is $G$-exact iff $$\Hom_{\cha{G}}(G,X) \rightarrowtail \Hom_{\cha{G}}(G,Y) \twoheadrightarrow \Hom_{\cha{G}}(G,Z)$$ is a short exact sequence of abelian groups, iff $$\prod_{n,i} \Hom(D^n(G_i),X) \rightarrowtail \prod_{n,i} \Hom(D^n(G_i),Y) \twoheadrightarrow \prod_{n,i} \Hom(D^n(G_i),Z)$$ is short exact, iff $\prod_{n,i} \Hom_{G}(G_i,X_n) \rightarrowtail \prod_{n,i} \Hom_{G}(G_i,Y_n) \twoheadrightarrow \prod_{n,i} \Hom_{G}(G_i,Z_n)$ is short exact, iff $X \rightarrowtail Y \twoheadrightarrow Z$ is degreewise $G$-exact (where here $G = \oplus G_i$) in $\cat{G}$.
\end{proof}

Being an exact category, $\cha{G}_G$ comes with a Yoneda Ext group, which in this case is the group of (equivalence classes of) degreewise $G$-exact sequences $Y \rightarrowtail Z \twoheadrightarrow X$, with addition defined by the Baer sum. We will denote this bifunctor by $\GExt^1_{\cha{G}}$, and note that for given chain complexes $X$ and $Y$, $\GExt^1_{\cha{G}}(X,Y)$ is a subgroup of the usual Yoneda $\Ext^1_{\cha{G}}(X,Y)$.  We sometimes will also call an element of $\GExt^1_{\cha{G}}(X,Y)$ a \textbf{degreewise $\boldsymbol{G}$-extension}. We also denote by $\GExt^1_{\cat{G}}$, the group of \textbf{$\boldsymbol{G}$-extensions} in the ground category $\cat{G}_G$. We have the following $G$-versions of standard isomorphisms.

\begin{lemma}\label{lemma-disk adjunctions}
Let $A \in \cat{G}$ and $X \in \cha{G}$. Then we have the following natural isomorphisms.
\begin{enumerate}
\item $\GExt^1_{\cha{G}}(D^n(A),X) \cong \GExt^1_{\cat{G}}(A,X_n)$

\item $\GExt^1_{\cha{G}}(X,D^{n+1}(A)) \cong \GExt^1_{\cat{G}}(X_n,A)$
\end{enumerate}

\end{lemma}

\begin{proof}
The point is that the standard isomorphisms take degreewise $G$-extensions to $G$-extensions. For example, for (1), the standard mapping $\Ext^1_{\cha{G}}(D^n(A),X) \xrightarrow{} \Ext^1_{\cat{G}}(A,X_n)$ takes a short exact sequence $0\xrightarrow{} X \xrightarrow{} Z \xrightarrow{} D^n(A) \xrightarrow{} 0$ to
$0\xrightarrow{} X_n \xrightarrow{} Z_n \xrightarrow{} A \xrightarrow{} 0$. Its inverse is formed by taking an extension $0\xrightarrow{} X_n \xrightarrow{} Z \xrightarrow{} A \xrightarrow{} 0$ and forming the pushout of $X_{n-1} \xleftarrow{d_n} X_n \xrightarrow{} Z$. Since pushouts of $G$-monomorphisms are again $G$-monomorphisms, we see that the isomorphisms restrict nicely between $G$-extensions. This shows (1). The isomorphism (2) is dual, using that pullbacks of $G$-epimorphisms are again $G$-epimorphisms.
\end{proof}

There is one more exact category that will be of use.
We denote by $\cha{G}_{dw}$ the category of all chain complexes along with the proper class of all degreewise split short exact sequences. We denote its Yoneda Ext bifunctor by $\Ext^1_{dw}$. We note that we have subgroup containments $$\Ext^1_{dw}(X,Y)  \subseteq \GExt^1_{\cha{G}}(X,Y) \subseteq \Ext^1_{\cha{G}}(X,Y) ,$$ and we have the following well-known connection between $\Ext^1_{dw}$ and the functor $\homcomplex$.

\begin{lemma}\label{lemma-homcomplex-basic-lemma}
For chain complexes $X$ and $Y$, we have isomorphisms:
$$\Ext^1_{dw}(X,\Sigma^{(-n-1)}Y) \cong H_n \homcomplex(X,Y) =
\cha{G}(X,\Sigma^{-n} Y)/\sim$$ In particular, for chain complexes $X$ and $Y$, $\homcomplex(X,Y)$ is
exact iff for any $n \in \Z$, any chain map $f \mathcolon \Sigma^nX \xrightarrow{} Y$ is
homotopic to 0 (or iff any chain map $f \mathcolon X \xrightarrow{} \Sigma^nY$ is homotopic
to 0).
\end{lemma}

We note also that the functor
$\homcomplex(X,-) \mathcolon \cha{G} \xrightarrow{} \textnormal{Ch}(\Z)$ takes degreewise $G$-exact sequences to short exact sequences if each $X_{n}$ is $G$-projective. Similarly the contravariant functor $\homcomplex(-,Y)$ sends degreewise $G$-exact sequences to short exact sequences if each $Y_{n}$ is G-injective.

\subsection{$\boldsymbol{G}$-acyclic complexes.}

Following definition~\cite[Definition~10.1]{buhler-exact categories}), an acyclic chain complex with respect to the exact structure $\cat{G}_G$ ought to be a chain complex $X$ for which its differentials each factor as $X_n \twoheadrightarrow Z_{n-1}X \hookrightarrow X_{n-1}$ in such a way that $Z_nX \hookrightarrow X_n \twoheadrightarrow Z_{n-1}X$ is $G$-acyclic. We will call such a complex \emph{$G$-acyclic} (or \emph{$G$-exact}).

\begin{lemma}\label{lemma-properties G-exact}
We have the following properties of $G$-acyclic complexes.
\begin{enumerate}
\item Let $X$ be a chain complex. Then the following are equivalent:
\begin{enumerate}
\item $X$ is $G$-acyclic.
\item $X$ is exact and $Z_nX \subseteq_G X_n$ is a $G$-subobject for each $n$.
\item $\Hom_{\cat{G}}(G,X)$ is exact.
\item Each $\Hom_{\cat{G}}(G_i,X)$ is exact.
\end{enumerate}
Note in particular that any $G$-acyclic complex is exact in the usual sense.
\item If $X$ is contractible, meaning $1_X \sim 0$, then $X$ is $G$-acyclic.
\item The class of $G$-acyclic complexes is thick in $\cha{G}_G$. That is, it is closed under retracts and for any exact $X \rightarrowtail Y \twoheadrightarrow Z$ in $\cha{G}_G$, if two out of three terms are $G$-acyclic then so is the third.
\end{enumerate}
\end{lemma}

\begin{proof}
For (1), we clearly have $\text{(a)} \Leftrightarrow \text{(b)} \Rightarrow \text{(c)} \Rightarrow \text{(d)}$. (d) and (c) are equivalent because $\Hom_{\cat{G}}(G,X) \cong \prod_{i \in I} \Hom_{\cat{G}}(G_i,X)$ (and a product of exact complexes is exact in \textbf{Ab}). Using Lemma~\ref{lemma-generators reflect exactness} we see (c) implies (b).

For (2), recall that having $1_X \sim 0$ means there exists maps $\{s_n : X_n \xrightarrow{} X_{n+1}\}$ such that $sd+ds=1$. Applying the additive functor $\Hom_{\cat{G}}(G,-)$ to this equation shows that $\Hom_{\cat{G}}(G,X)$ is also contractible. In particular it is exact.

For (3), note that if $X \hookrightarrow Y \twoheadrightarrow Z$ is a short exact sequence in $\cha{G}_G$, then since it is degreewise $G$-exact we get a short exact sequence of complexes of abelian groups $0 \xrightarrow{} \Hom_{\cat{G}}(G,X) \xrightarrow{} \Hom_{\cat{G}}(G,Y) \xrightarrow{} \Hom_{\cat{G}}(G,Z) \xrightarrow{} 0$. If any two out of three of these are exact then so is the third. For retracts, note that any additive functor preserves retracts. So this is true since a retract of an exact complex of abelian groups is again an exact complex.
\end{proof}

\subsection{Projectives in $\boldsymbol{\text{Ch}(\cat{G})_G}$.}

Here we classify the projective objects of $\cha{G}_G$.

\begin{lemma}\label{lemma-complexes that are projective}
Call a chain complex $X$ in $\cha{G}$ a \textbf{$\boldsymbol{G}$-projective} complex if it is projective in the exact category $\cha{G}_G$. The following are equivalent:
\begin{enumerate}
\item $X$ is $G$-projective.
\item $X$ is $G$-acyclic with each $Z_nX$ a $G$-projective.
\item $X$ is isomorphic to a split exact complex with $G$-projective components. That is, $X \cong \oplus_{n \in \Z} D^n(P_n)$ where each $P_n$ is a $G$-projective.
\item $X$ is a contractible complex with each $X_n$ $G$-projective.
\end{enumerate}
\end{lemma}

\begin{proof}
Using part~(3) of Corollary~\ref{cor-G-projective generator} and~\cite[Corollary~2.7]{gillespie-recoll2} we can find, for any chain complex $X$, a $G$-epimorphism $\oplus_{n \in \Z} D^n(P_n) \twoheadrightarrow X$ in which each $P_n$ is $G$-projective. If $X$ is $G$-projective, then this is a split epi. Then (2),(3), and~(4) all follow and are equivalent by standard arguments. On the other hand, the isomorphism $\GExt^1_{\cha{G}}(D^n(A),X) \cong \GExt^1_{\cat{G}}(A,X_n)$ of Lemma~\ref{lemma-disk adjunctions} tells us that a disk $D^n(A)$ is $G$-projective if and only if $A$ is $G$-projective in $\cat{G}_G$. Moreover, in any exact category a direct sum is projective if and only if each summand is projective by~\cite[Corollary~11.7]{buhler-exact categories}.
\end{proof}

\subsection{The $\boldsymbol{G}$-derived category.}

We now construct the $G$-derived category by putting a cofibrantly generated ``projective'' model structure on $\cha{G}_G$. The model structure follows as a Corollary to the next theorem. The proof relies on Quillen's small object argument. We refer to the version in~\cite[Theorem~2.1.14]{hovey-model-categories} and in particular we refer the reader there for the definition of the notation $I\text{-cell}$ and $I\text{-inj}$. We also refer the reader to~\cite{hovey} for the language of cotorsion pairs.

We define two sets of maps which will respectively be the \emph{generating cofibrations} and \emph{generating trivial cofibrations}:
$$I = \{0 \rightarrowtail D^n(G_i)\} \cup \{S^{n-1}(G_i) \rightarrowtail D^n(G_i)\} \ , \ \ \ \text{and} \ \ \ J = \{0 \rightarrowtail D^n(G_i)\}.$$
We also define the following set of objects which will cogenerate the cotorsion pair: $$\class{S} = \{D^n(G_i)\} \cup \{S^n(G_i)\}.$$
Note that $\class{S} = \cok{I} = \{\cok{i} \,|\, i \in I \}$. We leave it to the reader to check the easy fact that a chain complex $X$ satisfies $X \in \rightperp{\class{S}}$  if and only if $(X \xrightarrow{} 0) \in I\text{-inj}$. (The ``perp'' here is taken with respect to the degreewise $G$-exact sequences. So use $\GExt^1_{\cha{G}}$ and the fact that the $D^n(G_i)$ are $G$-projective.)

\begin{theorem}\label{them-projective model for G-derived}
Let $\cat{G}$ be any Grothendieck category with a generator $G = \oplus G_i$. Let $\class{W}$ denote the class of all $G$-acyclic complexes. Then the set $\class{S} = \{D^n(G_i)\} \cup \{S^n(G_i)\}$ cogenerates a cotorsion pair $(\class{P},\class{W})$ in the exact category $\cha{G}_G$ with the following properties.
\begin{enumerate}
\item $(\class{P},\class{W})$ is complete. In fact, for any chain complex $X$ there is a $G$-exact sequence $W \rightarrowtail P \twoheadrightarrow X$ where $W \in \class{W}$ and $P \in \class{P}$ is a transfinite (degreewise-split) extension of $\class{S}$. In particular, each $P_n$ is a direct sum of copies of the $G_i$.
    \item $P \in \class{P}$ if and only if $P$ is a retract of a  transfinite (degreewise-split) extension of $\class{S}$. We will call a complex in $\class{P}$ a \textbf{semi-$\boldsymbol{G}$-projective} complex.
\item $\class{W}$ is thick and $\class{P} \cap \class{W}$ coincides with the class of projective complexes in $\cha{G}_G$. (See Lemma~\ref{lemma-complexes that are projective}.)
\end{enumerate}
\end{theorem}

For $\cat{G}$ = $R$-Mod and $G=R$, this recovers the usual projective model structure on $\ch$ where the cofibrant complexes are the DG-projective complexes. Some authors call these complexes \emph{semiprojective}, and since DG-$G$-projective looks odd we use semiprojective.

Our proof of Theorem~\ref{them-projective model for G-derived} is based on the proof of~\cite[Theorem~6.5]{hovey}. Indeed for the case when $\cat{G}$ is locally finitely presentable (that is, $G = \oplus G_i$ where the $G_i$ are finitely presented), we only need the first paragraph of the proof below, combined with Corollary~\ref{cor-transfinite compositions and extensions of G-monos} and~\cite[Theorem~6.5]{hovey}.

\begin{proof}
Since each $G_i$ is $G$-projective we have an equality
$$\GExt^1_{\cha{G}}(S^n(G_i),X) = \Ext^1_{dw}(S^n(G_i),X).$$
So $X \in \rightperp{\{S^n(G_i)\}}$ if and only if for each $n$ we have vanishing of
$$\Ext^1_{dw}(S^n(G_i),X) = H_{n-1}\homcomplex(S^0(G_i),X) = H_{n-1}\Hom_{\cat{G}}(G_i,X).$$ So $X  \in \rightperp{\{S^n(G_i)\}}$ if and only if $X$ is $G$-acyclic. So indeed $\class{S}$ cogenerates a cotorsion pair $(\class{P},\class{W})$ in the exact category $\cha{G}_G$.

To show this cotorsion pair is complete we apply the small object argument from~\cite[Theorem~2.1.14]{hovey-model-categories}. We can do this since every object in a Grothendieck category is \emph{small}. The small object argument provides, for a given map $X \xrightarrow{} Y$, a functorial factorization $X \xrightarrow{} Z \xrightarrow{} Y$ where $(X \xrightarrow{} Z) \in I\text{-cell}$ and $(Z \xrightarrow{} Y) \in I\text{-inj}$. So we now pause to better understand $I\text{-cell}$ and $I\text{-inj}$.

\

\noindent \underline{Claim}: If $(p : Z \xrightarrow{} Y) \in I\text{-inj}$, then $p$ is a degreewise $G$-epimorphism with $G$-acyclic kernel. For completeness, we include the following direct proof of the claim. However,  we note that  
$I\text{-inj}$ can also be characterized by applying~\cite[Propositions~1.3--1.6]{hovey-sheaves}. Indeed take the set $\class{M}$ in~\cite[Definition~1.1]{hovey-sheaves} to be $\class{M} = \{0 \rightarrowtail G_i \} \cup \{0 \rightarrowtail 0 \}$. 

To prove the claim, say we have such a $p : Z \xrightarrow{} Y$ in $I\text{-inj}$. Then for each $n$ and $i$ we have a lift in the diagram
$$\begin{CD}
0    @>>> Z \\
@VVV             @VVpV   \\
 D^n(G_i) @>>> Y  \\
\end{CD}$$ This implies that each $p_n$ is a $G$-epimorphism.  So now we have $K \rightarrowtail Z \twoheadrightarrow Y$ is $G$-exact where $K = \ker{p}$. It is left to show $K$ is $G$-acyclic. For \emph{any} set of maps $I$, it is an easy exercise to check that $I\text{-inj}$ is closed under pullbacks. Since $K \xrightarrow{} 0$ lies in the pullback square
$$\begin{CD}
K    @>>> Z \\
@VVV             @VVpV   \\
0 @>>> Y  \\
\end{CD}$$ we see $(K \xrightarrow{} 0) \in I\text{-inj}$. But as pointed out above the statement of the theorem, this is equivalent to saying $K \in \rightperp{\class{S}}$. So $K$ is $G$-acyclic.

\

\noindent \underline{Claim}: If $(f : X \xrightarrow{} Z) \in I\text{-cell}$, then $f$ is a degreewise split monomorphism with cokernel a transfinite extension of $\class{S}$.

To prove this, say $f : X \xrightarrow{} Z$ is in $I\text{-cell}$. By definition, $f$ is a transfinite composition of pushouts of maps of the form $0 \rightarrowtail D^n(G_i)$ or $S^{n-1}(G_i) \rightarrowtail D^n(G_i)$. Note that such pushouts are necessarily degreewise split monomorphisms whose cokernels are in $\class{S}$. This means $(f : X \xrightarrow{} Z) = (X_0 \xrightarrow{f} \varinjlim_{\alpha < \lambda} X_{\alpha})$ is a transfinite (degreewise split) extension of $X=X_0$ by $\class{S}$. Since transfinite extensions of split monomorphisms are again split monomorphisms, we conclude that $f$ too is a degreewise split monomorphism. We now look at $\cok{f}$. Since direct limits are exact we have a short exact sequence $X_0 \rightarrowtail \varinjlim X_{\alpha} \twoheadrightarrow \varinjlim (X_{\alpha}/X_0)$. In particular, $\cok{f} \cong \varinjlim (X_{\alpha}/X_0)$ is a transfinite extension of
$$0 \rightarrowtail X_1/X_0 \rightarrowtail X_2/X_0 \rightarrowtail X_3/X_0 \rightarrowtail \cdots \rightarrowtail X_{\alpha}/X_0 \rightarrowtail \cdots$$ But $(X_{\alpha+1}/X_0)/(X_{\alpha}/X_0) \cong X_{\alpha+1}/X_{\alpha} \in \class{S}$. This proves $f$ is a degreewise split monomorphism with cokernel a transfinite extension of $\class{S}$.

We now can prove that $(\class{P},\class{W})$ is complete. So suppose $Y$ is an arbitrary chain complex and use the small object argument to factor $0 \xrightarrow{} Y$ as $0 \xrightarrow{} Z \xrightarrow{p} Y$ where $0 \xrightarrow{} Z  \in I\text{-cell}$ and $Z \xrightarrow{p} Y$ is in $I\text{-inj}$. Then $K \rightarrowtail Z \twoheadrightarrow Y$ is a degreewise $G$-exact sequence with $K$ a $G$-acyclic complex. Also, $Z$ must be a transfinite extension of $\class{S}$. But by \cite[Lemma~6.2]{hovey} (taking the $G$-exact sequences as the \emph{proper class} of short exact sequences) we have that $\class{P}$ is closed under retracts and transfinite extensions. Therefore $Z \in \class{P}$ and $(\class{P},\class{W})$ has enough projectives in the way we claim in (1). To see that $(\class{P},\class{W})$ has enough injectives we instead factor
$X \xrightarrow{} 0$ as $X \xrightarrow{f} Z \xrightarrow{} 0$ where $X \xrightarrow{f} Z  \in I\text{-cell}$ and $Z \xrightarrow{} 0$ is in $I\text{-inj}$. Then $f$ is a degreewise split monomorphism (so a $G$-mono) with $\cok{f} \in \class{P}$, and $Z \in \class{W}$.

Next, statement (2). As mentioned above, $\class{P}$ is closed under retracts. Statement (2) is then a result of the following observation: Given $Q \in \class{P}$, write a $G$-exact sequence $W \rightarrowtail P \twoheadrightarrow Q$ where $W \in \class{W}$ and $P \in \class{P}$ is a transfinite (degreewise-split) extension of $\class{S}$. This $G$-exact sequence is an element of $\GExt^1_{\cha{G}}(Q,W) = 0$. So it splits and $Q$ is a retract of $P$ as desired.

For (3), we see from Lemma~\ref{lemma-properties G-exact} that $\class{W}$ is thick and contains all contractible complexes. So in particular $\class{W}$ contains the projective objects of $\cha{G}_G$ by Lemma~\ref{lemma-complexes that are projective}. Since $(\class{P},\class{W})$ is complete with $\class{W}$ thick and containing the projectives, the result follows by the argument in~\cite[Proposition~3.4]{bravo-gillespie-hovey}.
\end{proof}

In the language of~\cite{gillespie-exact model structures} and~\cite{gillespie-recoll}, parts (1) and (3) of the above Theorem say that $(\class{P},\class{W})$ is a projective cotorsion pair in $\cha{G}_G$. Such a cotorsion pair is equivalent to a model structure for which every object is fibrant. The  following corollary records some basic facts about this model structure.

\begin{corollary}\label{cor-projective model for G-derived}
Let $\cat{G}$ be any Grothendieck category with a generator $G = \oplus G_i$. Then there is a model structure on $\cha{G}$ which we call the \textbf{$\boldsymbol{G}$-projective model structure} whose trivial objects are the $G$-acyclic complexes. This gives us a model for the \textbf{$\boldsymbol{G}$-derived category}, which we denote by $\class{D}(G)$. The model structure satisfies the following:
\begin{enumerate}
\item The fibrations are precisely the $G$-epimorphisms. That is, the chain maps which are $G$-epimorphisms in each degree.
\item The trivial fibrations are the $G$-epimorphisms with $G$-acyclic kernel.
\item The cofibrations are the degreewise split monomorphisms whose cokernel is a semi-$G$-projective complex.
\item The trivial cofibrations are the split monomorphisms whose cokernel is a $G$-projective complex.
\item The weak equivalences are the \textbf{$\boldsymbol{G}$-homology isomorphisms}. That is, the chain maps $f : X \xrightarrow{} Y$ for which $\Hom_{\cat{G}}(G,f) : \Hom_{\cat{G}}(G,X) \xrightarrow{} \Hom_{\cat{G}}(G,Y)$ is a homology isomorphism.
\item The model structure is cofibrantly generated. The sets $I$ and $J$ from above are respectively the generating cofibrations and generating trivial cofibrations. Thus $\class{D}(G)$ is well generated in the sense of~\cite{neeman-well generated}.
\item If each $G_i$ is finitely presented, then the model structure is finitely generated and so in this case $\class{D}(G)$ is compactly generated.

\end{enumerate}
\end{corollary}

\begin{proof}
In the exact category $\cha{G}_G$, we have the complete cotorsion pair $(\class{P},\class{W})$. We also have the complete cotorsion pair $(\class{Q},\class{A})$ where $\class{Q}$ is the class of $G$-projective complexes of Lemma~\ref{lemma-complexes that are projective} and $\class{A}$ is the class of all complexes. Theorem~\ref{them-projective model for G-derived} along with the main theorem of~\cite{hovey} imply that we automatically have the model structure with (trivial) fibrations and (trivial) cofibrations as described.

In the correspondence between cotorsion pairs and model structures, the weak equivalences are precisely the maps which factor as a trivial cofibration followed by a trivial fibration. We wish to see that such maps are exactly the $G$-homology isomorphisms. First, given any $f : X \xrightarrow{} Y$, lets denote the the composite functor $H_n[\Hom_{\cat{G}}(G,f)]$ simply by $H_nf_*$. Using the model structure we can apply the factorization axiom and write $f = pi$ where $p$ is a fibration and $i$ is a trivial cofibration. We have $H_nf_* = H_np_* \circ H_ni_*$. Since $i$ is a split monomorphism with $G$-projective (so $G$-acyclic) cokernel, we see $H_ni_*$ is an isomorphism. So $H_nf_*$ is an isomorphism if and only if $H_np_*$ is an isomorphism. Since $p$ is a $G$-epimorphism, we see $H_np_*$ is an isomorphism (for all $n$) if and only if $\ker{p}$ is $G$-acyclic. That is, iff $p$ is a trivial fibration. We have now shown that $f$ factors as a trivial cofibration followed by a trivial fibration iff $H_nf_*$ is an isomorphism for all $n$.

It is easy to see that $J\text{-inj}$ is the class of $G$-epimorphisms. This means that $J$ is the set of generating trivial cofibrations. We also showed that everything in $I\text{-inj}$ is a $G$-epimorphism with $G$-acyclic kernel. So it is left to show that every $G$-epimorphism with $G$-acyclic kernel is in $I\text{-inj}$. So let $X \xrightarrow{p} Y$ be a $G$-epimorphism with kernel $K \in \class{W}$. Being a $G$-epimorphism we know that there is a lift in any diagram of the form
$$\begin{CD}
0    @>>> X \\
@VVV             @VVpV   \\
 D^n(G_i) @>>> Y  \\
\end{CD}$$
So all we need to show is that there is a lift for any diagram
$$\begin{CD}
 S^{n-1}(G_i)   @>>f> X \\
@ViVV             @VVpV   \\
 D^n(G_i) @>>g> Y  \\
\end{CD}$$
But again, we may start by finding an $D^n(G_i) \xrightarrow{h} X$ such that $ph = g$. We check that $(f-hi)$ lands in the kernel $K$. Now since $\GExt^1_{\cat{G}}(S^{n-1}(G_i),K)=0$, we see that the map $(f-hi)$ extends to some $D^n(G_i) \xrightarrow{\psi} K$. That is, $\psi i = (f-hi)$. So now we check that $(h+j\psi)$, where $j : K \rightarrowtail X$ is the desired lift. (i) $p(h+j\psi) = ph +0 = ph = g$. (ii) $(h+j\psi)i = hi+j\psi i = hi+(f-hi) = f$.

Since we have a cofibrantly generated model structure on a locally presentable (pointed) category, a main result from~\cite{rosicky-brown representability combinatorial model srucs} assures us that $\class{D}(G) = \Ho{\cha{G}}$ is well generated in the sense of~\cite{neeman-well generated} and~\cite{krause-well generated}. In the case that $G = \oplus G_i$ has each $G_i$ finitely presented, then the $G_i$ are \emph{finite} in the sense of~\cite[Section~7.4]{hovey-model-categories}. We then see that our model structure is \emph{finitely generated} and so~\cite[Corollary~7.4.4]{hovey-model-categories} tells us that $\class{D}(G) = \Ho{\cha{G}}$ has a set of small weak generators. In other words, it is compactly generated.
\end{proof}

\begin{remark}
Recall that by definition, a set $\class{S}$ of objects in a triangulated category such as $\class{D}(G)$ is called a set of \emph{weak generators} if $X = 0$ in $\class{D}(G)$ if and only if $\class{D}(G)(\Sigma^nS,X) = 0$ for all $n$ and $S \in \class{S}$. It is easy to see directly that $\{G_i = S^0(G_i)\}$ is a set of weak generators for $\class{D}(G)$. Indeed we wish to see that $X$ is $G$-acyclic if and only if $\class{D}(G)(S^n(G_i),X) = 0$ for all $n$ and $i$. But in the $G$-projective model structure we have that each $S^n(G_i)$ is cofibrant and every $X$ is fibrant, so we get that $\class{D}(G)(S^n(G_i),X) \cong \cha{G}(S^n(G_i),X)/\hspace{-.05in}\sim$ and the homotopy relation $\sim$ is the usual relation of chain homotopic maps. So it all boils down to checking that $X$ is $G$-acyclic if and only if $\cha{G}(S^n(G_i),X)/\hspace{-.05in}\sim \ = 0$ for all $n$ and $i$. But this is clear upon noting that $\cha{G}(S^n(G_i),X)/\hspace{-.05in}\sim \  \cong H_n[\Hom_{\cat{G}}(G_i,X)]$ and  referring to Lemma~\ref{lemma-properties G-exact}.

\end{remark}

\subsection{Computation of $\boldsymbol{\GExt^n_\cat{G}}$.}\label{subsec-GExt computation}

We have already seen an obvious analogy: $G$ is to $\class{G}_G$ as $R$ is to $R$-Mod. This analogy extends to the calculation of $\GExt^n_{\cat{G}}(A,B)$, as the existence of the $G$-projective model structure formalizes the fact that one can do homological with respect to $G$. In more detail, according to Corollary~\ref{cor-G-projective generator}, given any $A \in \cat{G}$, we may take a $G$-projective resolution
$$\class{P} \twoheadrightarrow A \ \equiv \ \cdots \rightarrow P_2 \rightarrow P_1 \rightarrow P_0 \twoheadrightarrow A.$$ By this we mean it is $G$-acyclic and each $P_n$ is $G$-projective. Then all the usual definitions and theorems hold for $G$-projective resolutions. For example, they are unique up to chain homotopy and one can define $\GExt^n_{\cat{G}}(A,B)$ via such resolutions. We obtain long exact sequences, starting with $G$-exact sequences, etc. Moreover $\GExt^n_{\cat{G}}(A,B)$ can alternately be defined using Yoneda's method: as equivalence classes of $G$-exact sequences $B \rightarrowtail L_1 \cdots \rightarrow L_n \twoheadrightarrow A$. (See also~\cite[Sections~1.2 and~2.1]{hovey-christensen-relative hom alg}; it is easy to see that the $G$-projectives and $G$-epimorphisms form a \emph{projective class}.)

Our point here is that for a $G$-projective resolution $\class{P} \twoheadrightarrow A$, we have a $G$-exact sequence of chain complexes
$K \rightarrowtail \class{P} \twoheadrightarrow S^0(A)$, where $K = \ker{(\class{P} \twoheadrightarrow S^0(A))}$. Moreover $K$ is $G$-acyclic and $\class{P}$ is semi-$G$-projective (since it is built up as a transfinite extension by consecutively attaching the semi-$G$-projective spheres $S^0(P_0)$, $S^1(P_1)$, $S^2(P_2)$, ...) So $\class{P}$ is a cofibrant replacement of $S^0(A)$ in the $G$-projective model structure. Hence using the fundamental theorem of model categories we have $$\class{D}(G)(A,\Sigma^nB) = \cha{G}(\class{P},S^n(B))/\hspace{-.05in}\sim \ = \ H^n[\Hom_{\cat{G}}(\class{P},B)] = \GExt^n_{\cat{G}}(A,B).$$

\subsection{The $\boldsymbol{\lambda}$-pure derived category.}\label{subsec-lambda pure derived cat}

In~\cite[Section~5.3]{hovey-christensen-relative hom alg} we see the construction of a model structure for the pure derived category of a ring $R$ and a canonical adjunction between the pure derived category of $R$ and the usual derived category $\class{D}(R)$. We describe now a natural extension of this fact to the $G$-derived category.

In any Grothendieck category $\cat{G}$, all objects are $\lambda$-presentable for some regular cardinal $\lambda$. In particular, for any choice of generator $G =\oplus G_i$ there is a $\lambda$ such that all the $G_i$ are $\lambda$-presentable. It follows that $\cat{G}$ is locally $\lambda$-presentable. (See Appendix~\ref{appendix-lambda pure} and~\cite[page 22]{adamek-rosicky} for language in this Subsection.) In fact, we see from~\cite[page 22]{adamek-rosicky} that $\class{G}$ is locally $\lambda$-presentable if and only if it has a generating set consisting of $\lambda$-presentable objects. Moreover, the category of all $\lambda$-presentable objects in $\cat{G}$ has a small skeleton. So we can find a set $\{\Lambda_i\}$ of representatives from each isomorphism class, and we set $\Lambda = \oplus \Lambda_i$. Since $\{\Lambda_i\}$ contains $\{G_i\}$, it is also a generating set for $\cat{G}$. From Lemma~\ref{lemma-properties G-exact} and Proposition~\ref{prop-lambda pure short exact sequences} the $\Lambda$-acyclic complexes are characterized as the exact complexes $X$ for which each $Z_nX \subseteq X_n$ is $\lambda$-pure. Such complexes are called \textbf{$\boldsymbol{\lambda}$-pure acyclic}. We call $\class{D}(\Lambda)$ the \textbf{$\boldsymbol{\lambda}$-pure derived category} of $\cat{G}$, and its model structure from Corollary~\ref{cor-projective model for G-derived} we call the \textbf{$\boldsymbol{\lambda}$-pure projective model structure}. The extension groups of Subsection~\ref{subsec-GExt computation} we denote by $\lambda$-PExt${}^n_{\cat{G}}$. We easily get the following.

\begin{corollary}\label{cor-lambda derived cat}
Let $\cat{G}$ be any Grothendieck category with $G$ and $\Lambda$ as above.
There is a canonical functor $\class{D}(\Lambda) \xrightarrow{} \class{D}(G)$ that is the identity on objects from the $\lambda$-pure derived category to the $G$-derived category. It induces  a map $\lambda$-\textnormal{PExt}${}^n_{\cat{G}}(A,B) \rightarrow \GExt^n_{\cat{G}}(A,B)$ which is natural in $A,B \in \cat{G}$. Moreover, $\class{D}(\Lambda) \xrightarrow{} \class{D}(G)$ admits a left adjoint.
\end{corollary}

\begin{proof}
First note that the identity functor $\cha{G} \xrightarrow{\text{id}} \cha{G}$ is left adjoint to itself. Since $\{G_i\} \subseteq \{\Lambda_i\}$, the identity functor takes semi-$G$-projective complexes (complexes built from all the $S^n(G_i)$) to semi-$\Lambda$-projective complexes (complexes built from all the $S^n(\Lambda_i)$). Similarly it takes $G$-projective complexes (those built from the $D^n(G_i)$) to $\Lambda$-projective complexes (built from the $D^n(\Lambda_i)$). This directly leads us to conclude the identity functor is a left Quillen functor from the $G$-projective model structure to the $\lambda$-pure projective model structure. This automatically provides an adjunction $\class{D}(G) \xrightarrow{L(\text{id})} \class{D}(\Lambda)$, taking a complex $X$ to its semi-$G$-projective cofibrant replacement. Its right adjoint $\class{D}(\Lambda) \xrightarrow{R(\text{id})} \class{D}(G)$ is the identity on objects since every object is fibrant.
Since the functor $R(\text{id})$ is identity on objects, the functor provides, for all $A,B \in \cat{G}$, a natural map $\class{D}(\Lambda)(A,\Sigma^nB) \rightarrow \class{D}(G)(A,\Sigma^nB)$. But from Subsection~\ref{subsec-GExt computation} we see this translates to a natural map $\lambda$-\textnormal{PExt}${}^n_{\cat{G}}(A,B) \rightarrow \GExt^n_{\cat{G}}(A,B)$.
\end{proof}

\section{The injective model for locally finitely presentable categories}\label{sec-injective model struc for G-derived}

In the previous section, we constructed the $G$-derived category of any pair $(\cat{G},G)$ where $\cat{G}$ is a Grothendieck category and $G = \oplus G_i$ is a generator. We constructed a model structure for $\class{D}(G)$ in which the cofibrant complexes were built from $G$-projective objects. Our goal in this section is to construct a dual model structure for $\class{D}(G)$, whose fibrant complexes are based on the $G$-injective objects. In order to do this we need to assume each $G_i$ is finitely presented, or equivalently, that $\cat{G}$ is \textbf{locally finitely presentable} (= locally $\omega$-presentable as defined in Appendix~\ref{appendix-lambda pure}). Indeed from~\cite[Theorem~1.11]{adamek-rosicky} we have that $\cat{G}$ is locally finitely presentable if and only if $\cat{G}$ has a set of generators $\{G_i\}_{i \in I}$ for which each $G_i$ is finitely presented (= $\omega$-presented).
Having different models for the same category is often useful. For example, the existence of the injective model structure implies the two recollement situations presented in Section~\ref{sec-recollement of Krause for G-derived}.

\subsection{$\boldsymbol{G}$-homology in locally finitely presentable categories.}

For a chain complex $X$, we define its $G$-homology as $H_n[\Hom_{\cat{G}}(G,X)]$. So the $G$-homology vanishes if and only if $X$ is $G$-acyclic. Recall that in a general Grothendieck category, a product of acyclic complexes need not again be acyclic. This is the point of Grothendieck's (AB4*) axiom. However, Theorem~\ref{them-projective model for G-derived} tells us that the $G$-acyclic complexes are closed under products, since they are the right half of a cotorsion pair. The point of this subsection is to collect other useful properties that hold under the added assumption that each $G_i$ is finitely presented. These properties will be used to construct the injective model structure on $\cha{G}$.

\begin{lemma}\label{lemma-G-homology and direct limits}
Assume each $G_i$ is finitely presented. Up to a product, $G$-homology commutes with direct limits. That is, if $\{X_j\}_{j \in J}$ is a directed system of complexes, then $$H_n[\Hom_{\cat{G}}(G,\varinjlim_{j \in J} X_j)] \cong \prod_{i \in I} \varinjlim_{j \in J} H_n[\Hom_{\cat{G}}(G_i,X_j)] $$
If the set of generators $\{G_i\} = \{G_1, G_2, \cdots , G_n\}$ is finite, then since direct limits commute with finite products we have
$$H_n[\Hom_{\cat{G}}(G,\varinjlim_{j \in J} X_j)] \cong \varinjlim_{j \in J} H_n[\Hom_{\cat{G}}(G,X_j)]$$
\end{lemma}

\begin{proof} For complexes of abelian groups, homology commutes with products and direct limits. Also, the $G_i$ are assumed finitely presented, so we have isomorphisms:
$$H_n[\Hom_{\cat{G}}(G,\varinjlim_{j \in J} X_j)] \cong \prod_{i \in I} H_n[\Hom_{\cat{G}}(G_i,\varinjlim_{j \in J} X_j)] \cong \prod_{i \in I} \varinjlim_{j \in J} H_n[\Hom_{\cat{G}}(G_i,X_j)].$$
\end{proof}

\begin{proposition}\label{prop-direct limits}
Assume each $G_i$ is finitely presented. Then the following hold.
\begin{enumerate}
\item  The $G$-acyclic complexes are closed under direct limits.
\item  Direct limits of $G$-monomorphisms are again $G$-monomorphisms.
\end{enumerate}
\end{proposition}

\begin{proof}
The first statement follows from Lemma~\ref{lemma-G-homology and direct limits}.
For the second, suppose $f$ is a monomorphism sitting in an exact sequence $\class{E} : 0 \xrightarrow{} A \xrightarrow{f} B \xrightarrow{g} C \xrightarrow{} 0$ which happens to be a directed limit of $G$-exact sequences $\class{E}_j : 0 \xrightarrow{} A_j \xrightarrow{} B_j \xrightarrow{} C_j \xrightarrow{} 0$. Interpreting each $\class{E}_j$ as a $G$-acyclic chain complex the result follows from the first statement.
\end{proof}

By a \emph{transfinite composition} we mean a map of the form $X_0 \xrightarrow{f} \varinjlim X_{\alpha}$ where $X : \lambda \xrightarrow{} \cat{G}$ is a colimit-preserving functor and $\lambda$ is an ordinal. In this case $f$ is the transfinite composition of the $X_{\alpha} \xrightarrow{} X_{\alpha +1}$. If each of these $X_{\alpha} \rightarrowtail X_{\alpha+1}$ is a $G$-monomorphism then $f$ is a \emph{transfinite composition of $G$-monomorphisms}.
Furthermore, in this case we say that  $\varinjlim X_{\alpha}$ is a \emph{transfinite $G$-extension} of all the objects $X_0, X_{\alpha + 1}/X_{\alpha}$.

\begin{corollary}\label{cor-transfinite compositions and extensions of G-monos}
Assume each $G_i$ is finitely presented. Then the following hold.
\begin{enumerate}
\item The $G$-acyclic complexes are closed under transfinite $G$-extensions and direct sums.
\item An arbitrary transfinite composition of $G$-monomorphisms is again a $G$-monomorphism.
\end{enumerate}
\end{corollary}

\begin{proof}
The $G$-acyclic complexes are always closed under $G$-extensions by Lemma~\ref{lemma-properties G-exact}. So they are closed under transfinite $G$-extensions by Proposition~\ref{prop-direct limits}. Direct sums are special cases of transfinite $G$-extensions.

For the second statement, we first note that a finite composition ($\lambda = n \in \N$) of $G$-monomorphisms is again a $G$-monomorphism by part~(4) of Proposition~\ref{prop-G-exact sequences-properties}. For $\lambda = \omega$, we want the map $X_0 \xrightarrow{f_{\omega}} \varinjlim_{n<\omega} X_{n}$ to also be a $G$-monomorphism. So we want the short exact sequence $$\class{E} : \ \ \ 0 \xrightarrow{} X_0 \xrightarrow{f_{\omega}} \varinjlim_{n<\omega} X_{n} \xrightarrow{} (\varinjlim_{n<\omega} X_{n})/X_0 \xrightarrow{} 0$$ to be $G$-exact. But this is the direct limit of the short exact sequences $$\class{E}_n : \ \ \ 0 \xrightarrow{} X_0 \xrightarrow{f_n} X_n \xrightarrow{} X_n/X_0 \xrightarrow{} 0$$ and these are $G$-exact because this is the finite case $\lambda = n$. So the $\lambda = \omega$ case holds by Proposition~\ref{prop-direct limits}. We see the result follows by transfinite induction.
\end{proof}

\subsection{Complete cotorsion pairs.}

The result here is taken, with only a few small adjustments for our situation,  from the original source~\cite{hovey} . We again use the notion of a \emph{small} cotorsion pair from~\cite{hovey} as well as the notation $I\text{-cell}$ and $I\text{-inj}$ from~\cite{hovey-model-categories}.

\begin{proposition}\label{prop-cotorsion pairs in the G-exact cat}
Consider the $G$-exact category $\cat{G}_G$ in the case that each $G_i$ is finitely presented.
Then a cotorsion pair $(\class{F},\class{C})$ in $\cat{G}_G$ is cogenerated by a set $\class{S}$ if and only if it is small with generating monomorphisms the set
$$I = \{0 \rightarrowtail G_i \}_{i \in I} \cup \{K_S \rightarrowtail P_S \twoheadrightarrow S \}_{S \in \class{S}}.$$ Here we have chosen for each $S \in \class{S}$, a $G$-exact sequence $K_S \rightarrowtail P_S \twoheadrightarrow S$ with $P_S$ a $G$-projective object. Such a cotorsion pair $(\class{F},\class{C})$ satisfies each of the following:
\begin{enumerate}
\item $(\class{F},\class{C})$ is functorially complete.

\item $\class{F}$ consists precisely of retracts of transfinite $G$-extensions of $\class{S}$.

\item $I\text{-inj}$ is precisely the class of all $G$-epimorphisms with kernel in $\class{C}$.
\end{enumerate}

\end{proposition}

\begin{proof}
Note that we can find the $G$-exact sequences $K_S \rightarrowtail P_S \twoheadrightarrow S$ with each $P_S$ a $G$-projective by using Corollary~\ref{cor-G-projective generator}.
We see that the functors $\GExt^1_{\cha{G}}(P_S,-)$ and $\GExt^1_{\cha{G}}(G_i,-)$ vanish. So it is easy to see that $\class{S}$ cogenerates the cotorsion pair iff the given set $I$ forms a set of generating monomorphisms in the sense of~\cite[Definition~6.4]{hovey}.

By Corollary~\ref{cor-transfinite compositions and extensions of G-monos} we have that transfinite compositions of $G$-monomorphisms are again $G$-monomorphisms. So by~\cite[Theorem~6.5]{hovey} we get that $(\class{F},\class{C})$ is a functorially complete cotorsion pair. The proof there shows that $\class{F}$ consists precisely of retracts of transfinite $G$-extensions of objects in $\class{S}$.

It is left to see that $I\text{-inj}$ is precisely the class of all $G$-epimorphisms with kernel in $\class{C}$. Showing that everything in $I\text{-inj}$ is a $G$-epimorphism with kernel in $\class{C}$ is formally similar to the first claim in the proof of Theorem~\ref{them-projective model for G-derived}. The converse is similar to the argument given in the last paragraph of the proof of Corollary~\ref{cor-projective model for G-derived}. We leave the details to the reader.
\end{proof}

\begin{remark}\label{remark on cotorsion pairs of complexes}
We note that Proposition~\ref{prop-cotorsion pairs in the G-exact cat} applies not just to $\cat{G}_G$ but also to $\cha{G}_G$ by Lemma~\ref{lemma-generators in Ch(G)}. This is because each $D^n(G_i)$ is a finitely presented complex whenever each $G_i$ is finitely presented. In particular, any cotorsion pair in $\cha{G}_G$ that is cogenerated by a set is complete.
\end{remark}

\subsection{Injectives in $\boldsymbol{\cat{G}_G}$ and $\boldsymbol{\text{Ch}(\cat{G})_G}$.}\label{subsection-G-injective objects}

We need to show that the exact categories $\cat{G}_G$ and $\cha{G}_G$ have enough injective objects. Following our language for the projective case, we will call these objects \emph{$G$-injective}. We will use the theory of purity summarized in Appendix~\ref{appendix-lambda pure}. The appendix shows that when $\cat{G}$ is locally finitely presentable (= locally $\omega$-presentable) we have a well-behaved notion of \textbf{pure} (= $\omega$-pure) subobjects $P \subseteq X$ in $\cat{G}$. In particular, we get that pure monomorphisms are closed under directed colimits (= $\omega$-directed colimits) in $\cat{G}$ by Proposition~\ref{prop-lambda pure short exact sequences}.

Note that any pure exact sequence $0 \xrightarrow{} A \xrightarrow{} B \xrightarrow{} C \xrightarrow{} 0$ in $\cat{G}$ is automatically a $G$-exact sequence. This follows from Proposition~\ref{prop-lambda pure short exact sequences}, our assumption that each $G_i$ is finitely presented, and the fact that direct products of short exact sequences (of abelian groups) are still short exact sequences. In particular, any pure subobject is automatically a $G$-subobject.

\begin{remark}\label{remark-lambda presented equals lambda generated}
For \emph{any} Grothendieck category $\cat{G}$ there exist arbitrarily large regular cardinals $\lambda$ such that the $\lambda$-presented objects coincide with the $\lambda$-generated objects. The author thanks Ji\v{r}\'i Rosick\'y for providing the following reason for this statement: Let $\cat{G}_{\text{mono}}$ denote the category consisting of the same objects as $\cat{G}$ but with morphisms only the monomorphisms of $\cat{G}$. Then for any $\lambda$, the $\lambda$-presented objects of $\cat{G}_{\text{mono}}$ coincide exactly with the $\lambda$-generated objects of $\cat{G}$. Moreover we note  $\cat{G}_{\text{mono}}$ is an accessible category by~\cite[Local Generation Theorem~1.70]{adamek-rosicky}. The embedding functor $\cat{G}_{\text{mono}} \xrightarrow{} \cat{G}$ is an \emph{accessible functor} in the sense of~\cite[Definition~2.16]{adamek-rosicky}. Therefore, the Uniformization Theorem~\cite[Theorem~2.19 and Remark]{adamek-rosicky} applies which means there are arbitrarily large regular cardinals $\lambda$ for which this embedding is $\lambda$-accessible and preserves $\lambda$-presented objects. This means exactly that there exist arbitrarily large regular cardinals $\lambda$ such that the $\lambda$-presented objects coincide with the $\lambda$-generated objects. In fact, it follows from~\cite[Remark~2.20]{adamek-rosicky} that if $\gamma$ is a regular cardinal for which $\lambda \triangleleft \gamma$, that is $\lambda$ is \emph{sharply smaller} than $\gamma$ in the sense of~\cite[Definition~2.12]{adamek-rosicky}, then the $\gamma$-presented objects coincide with the $\gamma$-generated objects too.
\end{remark}

Note that for any $\gamma$ as in Remark~\ref{remark-lambda presented equals lambda generated} the notion of $\gamma$-presented (= $\gamma$-generated) becomes a substitute for ``cardinality $< \gamma$''. In particular, the class of $\gamma$-presented objects is closed under quotients and subobjects. We also have that, up to isomorphism, there is just a set of $\gamma$-presented objects.

\begin{setup}\label{setup-gamma}
We now specify for our locally finitely presentable category $\cat{G}$ a regular cardinal $\gamma$ which will be of use. We fix a regular cardinal $\gamma$ with each of the following properties:
\begin{enumerate}
\item The $\gamma$-presented objects coincide with the $\gamma$-generated objects.
\item Whenever we have a subobject $S \subseteq X$ where $S$ is $\gamma$-generated, there exists a pure subobject $P \subseteq X$ which is also $\gamma$-generated and which contains $S$.
\end{enumerate}
\end{setup}

Lets now justify why we can choose such a cardinal $\gamma$. First, from the above Remark~\ref{remark-lambda presented equals lambda generated}
we can find a regular cardinal $\lambda$ such that whenever $\gamma$ is a regular cardinal with $\lambda \triangleleft \gamma$, then the $\gamma$-presented objects coincide with the $\gamma$-generated objects. Since our category $\cat{G}$ is locally $\omega$-presentable it is also locally $\lambda$-presentable, (since $\omega \leq \lambda$ and~\cite[Remark~1.20]{adamek-rosicky}). So by~\cite[Theorem~2.33]{adamek-rosicky} we are guaranteed the existence of arbitrarily large regular cardinals $\gamma \triangleright \lambda$ with the following property: Whenever we have a subobject $S \subseteq X$ where $S$ is $\gamma$-generated, there exists a $\lambda$-pure subobject $P \subseteq X$ which is also $\gamma$-generated and which contains $S$. However, any $\lambda$-pure $P$ is also pure because $\omega \leq \lambda$ and~\cite[Remark~(3) pp.~85]{adamek-rosicky}. (However, we warn the reader that there is a misprint in~\cite[Remark~(3) pp.~85]{adamek-rosicky}. The inequality goes the other way.) But we are done.

The main purpose for constructing $\gamma$ in Setup~\ref{setup-gamma} is to use its properties~(i) and~(ii) to show that any $G$-acyclic complex is a transfinite $G$-extension of $\gamma$-presented $G$-acyclic complexes. Although perhaps overkill, we will also now use $\gamma$ to show that $\cat{G}_G$ has enough injectives.

\begin{proposition}\label{prop-G-exact categories have enough injectives}
Let $\gamma$ be as in Setup~\ref{setup-gamma} and let $\class{S}$ be a set of isomorphic representatives for the class of all $\gamma$-presented objects. Then $\class{S}$ cogenerates the injective cotorsion pair $(\class{A},\class{I})$ in $\cat{G}_G$. That is, $\class{A}$ consists of all objects of $\cat{G}$, while $\class{I} = \rightperp{\class{S}}$ is precisely the class of injective objects of $\cat{G}_G$. We call these objects \textbf{$\boldsymbol{G}$-injective}.
$(\class{A},\class{I})$ is complete, meaning $\cat{G}_G$ has enough $G$-injectives.
\end{proposition}

\begin{proof}
By Proposition~\ref{prop-cotorsion pairs in the G-exact cat} we know that $\class{S}$ cogenerates a complete cotorsion pair $(\leftperp{(\rightperp{\class{S}})}, \rightperp{\class{S}})$ where $\leftperp{(\rightperp{\class{S}})}$ consist precisely of retracts of transfinite $G$-extensions of $\class{S}$. Letting $\class{A}$ denote the class of all objects of $\cat{G}$ we will be done if we can show $\class{A} \subseteq \leftperp{(\rightperp{\class{S}})}$. By~\cite[Lemma~6.2]{hovey} it suffices to show that every object in $\class{A}$ is a transfinite $G$-extension of objects in $\class{S}$. But since each $G_i$ is finitely presented, we note that pure exact sequences are automatically $G$-exact. So it is enough to show that any object is a transfinite pure-extension of $\gamma$-presented objects.

So let $M$ be any given object.
First note that assuming $M \neq 0$, we can always find a nonzero pure subobject $P_0 \subseteq M$ with $P_0$ $\gamma$-presented. Assuming $P_0 \neq M$, we can do the same to $M/P_0$ to get a pure $P_1/P_0 \subseteq M/P_0$ with $P_1/P_0$ $\gamma$-presented. Assuming we are not done, we continue to construct a strictly increasing $0 \neq P_0 \subseteq P_1 \subseteq P_2 \subseteq \cdots$. Note each $P_n \subseteq M$ is pure by part~(3) of Proposition~\ref{prop-purity properties}. Then set $P_{\omega} = \cup_{n < \omega} P_n$ and note it is also pure by part~(4) of Proposition~\ref{prop-purity properties}. In this way we continue by transfinite induction to get $M = \cup_{\alpha < \lambda} P_{\alpha}$ for some $\lambda$ where each $P_{\alpha} \subseteq P_{\alpha+1}$ is pure.
\end{proof}

\begin{remark}
No matter what our choice is for the generator $G = \oplus_{i \in I} G_i$, it is the same set $\class{S}$ cogenerating the injective cotorsion pair $(\class{A},\class{I})$ (as long as each $G_i$ is finitely presented). But a different choice of generating set $\{G_i\}$ will of course change the proper class of short exact sequences in $\cat{G}_G$. Consequently, this changes the class $\rightperp{\class{S}}$ of $G$-injectives. (It of course also changes the $G$-projectives).
\end{remark}

Note that because of Lemma~\ref{lemma-generators in Ch(G)}, the above Proposition~\ref{prop-G-exact categories have enough injectives} also applies to the chain complex category $\cha{G}_G$. That is, there are enough $G$-injective complexes. As in Lemma~\ref{lemma-complexes that are projective} we have the following classification of $G$-injective complexes.

\begin{lemma}\label{lemma-complexes that are injective}
Call a chain complex $X$ in $\cha{G}$ a \textbf{$\boldsymbol{G}$-injective} complex if it is injective in the exact category $\cha{G}_G$. The following are equivalent:
\begin{enumerate}
\item $X$ is $G$-injective.
\item $X$ is $G$-acyclic with each $Z_nX$ a $G$-injective.
\item $X$ is isomorphic to a split exact complex with $G$-injective components. That is, $X \cong \oplus_{n \in \Z} D^n(I_n)$ where each $I_n$ is a $G$-injective.
\item $X$ is a contractible complex with each $X_n$ $G$-injective.
\end{enumerate}
We note that there are enough $G$-injective complexes. This follows from Proposition~\ref{prop-G-exact categories have enough injectives} and Lemma~\ref{lemma-generators in Ch(G)}.
\end{lemma}

\subsection{The injective model structure.}\label{subsec-inj model}

We now wish to construct an injective model structure for the $G$-derived category, assuming each $G_i$ is finitely presented.
The following lemma, which holds for arbitrary Grothendieck categories, will be used in the main proof. It is a generalization of~\cite[Lemma~V.3.3]{stenstrom}.

\begin{lemma}\label{lemma-surjective}
Let $\cat{G}$ be a locally $\lambda$-presentable Grothendieck category. Given an
epimorphism $g \mathcolon X \xrightarrow{} Y$ where $Y$ is
$\lambda$-generated, there exists a $\lambda$-generated subobject
$X' \subseteq X$ for which $g_{|X'} \mathcolon X' \xrightarrow{} Y$ is an
epimorphism.
\end{lemma}

\begin{proof}
Any locally $\lambda$-presentable Grothendieck category is also locally $\lambda$-generated. This means that, up to isomorphism, there is a set of $\lambda$-generated objects and that every object is a $\lambda$-directed union of its $\lambda$-generated subobjects. (The proof of this goes by writing the given object $X = \varinjlim X_i$ as a $\lambda$-directed colimit of $\lambda$-presented $X_i$. Then factor each $X_i \xrightarrow{} X$ as an epi followed by a mono. Each $\im{X_i}$ is $\lambda$-generated and $X_i$ is the $\lambda$-directed union of the $\im{X_i}$.) So we may write $X =
\sum_{i \in I}X_i$ as a $\lambda$-directed union of
$\lambda$-generated subobjects of $X$. Since $g$ is an epimorphism,
$Y = \sum_{i \in I} g(X_i)$, and this too is a $\lambda$-directed
union. Now we must have $Y = g(X_i)$ for some $i \in I$ since $Y$ is $\lambda$-generated. So $g_{|X_i}
\mathcolon X_i \xrightarrow{} Y$ is an epimorphism.
\end{proof}

Recall (see Lemma~\ref{lemma-properties G-exact}), that a chain complex $X$ is $G$-acyclic if and only if it is exact and each $Z_nX$ is a $G$-subobject of $X_n$. This means the inclusion $Z_nX \rightarrowtail X_n$ is a $G$-monomorphism, and we write $Z_nX \subseteq_G X_n$.

\begin{lemma}\label{lemma-finding G-exact subcomplexes}
Let $\gamma$ be as in Setup~\ref{setup-gamma}. Given any nonzero $G$-acyclic complex $E$ there exists a degreewise $G$-exact sequence $P \rightarrowtail E \twoheadrightarrow E/P$ where $P$ is a nonzero $G$-acyclic subcomplex with each $P_n$ $\gamma$-presented.

\end{lemma}

\begin{proof}
(Step 1) We first prove the following: For any given $n$ and exact $S \subseteq E$ with each $S_i$ $\gamma$-presented, there exists an exact $T \subseteq E$ satisfying the following:
\begin{enumerate}
\item $S \subseteq T$ and each $T_i$ is $\gamma$-presented.
\item $Z_nT \subseteq_G T_n$ is a $G$-subobject.
\item $S_n \subseteq P \subseteq T_n \subseteq E_n$ for some $G$-subobject $P \subseteq E_n$.
\end{enumerate}
Indeed as in Setup~\ref{setup-gamma} we can find a $\gamma$-presented pure $P \subseteq E_n$ containing $S_n$. Then set $T_{n-1} = S_{n-1} + d(P)$ and note that it is $\gamma$-presented and that $\ker{d|_{T_{n-1}}} = d(P)$. We set $T_{n-2} = S_{n-2}$, $T_{n-3} = S_{n-3}$, etc. going downward. This gives us a portion of a subcomplex we are building $$ \cdots P \xrightarrow{}  T_{n-1} \xrightarrow{} T_{n-2} \xrightarrow{} T_{n-3} \xrightarrow{} \cdots $$ which we note is exact in degrees $n-1$ and below. We wish to extend upwards to an exact complex.

Note that $\ker{d|_{P}}$ is also $\gamma$-presented. So there exists a $\gamma$-presented pure subobject $P' \subseteq Z_nE$ containing $\ker{d|_{P}}$. Now let $T_n = P + P'$, and note that we still have exactness in degrees $\leq n-1$ in the (still unfinished) subcomplex shown
$$ \cdots T_n \xrightarrow{}  T_{n-1} \xrightarrow{} T_{n-2} \xrightarrow{} T_{n-3} \xrightarrow{} \cdots $$ Moreover, since $\ker{d|_{T_n}} = P'$ is pure in $Z_nE$, it is a $G$-subobject $\ker{d|_{T_n}} = P' \subseteq_G Z_nE$. We also have $Z_nE \subseteq_G E_n$ by assumption, and so from part~(1) of Proposition~\ref{prop-G-purity properties} we have $\ker{d|_{T_n}} = P' \subseteq_G E_n$. But then from part~(2) of Proposition~\ref{prop-G-purity properties} we have $\ker{d|_{T_n}} = P' \subseteq_G T_n$. (Here we have arranged conditions (2) and (3) in the subcomplex $T$ that we are constructing.)

Now since $P'$ is $\gamma$-presented, we can use Lemma~\ref{lemma-surjective} to find a $\gamma$-presented
subobject $S'_{n+1} \subseteq E_{n+1}$ for which $d_{|S'_{n+1}}
\mathcolon S'_{n+1} \xrightarrow{} P'$ is an epimorphism. We set $T_{n+1} = S_{n+1} + S'_{n+1}$ and note that
$$ \cdots T_{n+1} \xrightarrow{} T_n \xrightarrow{}  T_{n-1} \xrightarrow{} T_{n-2} \xrightarrow{} T_{n-3} \xrightarrow{} \cdots $$
is now exact in degrees $n$ and below. Repeatedly using Lemma~\ref{lemma-surjective} in this way we can continue upward to obtain an exact subcomplex $T \subseteq E$ which contains $S$, which has each $T_i$ $\gamma$-presented, has $Z_nT = P' \subseteq_G T_n$, and has $S_n \subseteq P \subseteq T_n \subseteq E_n$ where $P \subseteq_G E_n$.

(Step 2) We now complete the proof. For the construction just described in (Step 1), lets say that the complex $T$ was obtained by applying a ``degree $n$ operation to $S$''. Start by first finding \emph{any} nonzero exact complex $S \subseteq E$ with each $S_i$ $\gamma$-presented, and with this $S$ apply a ``degree 0 operation to $S$'' to obtain a $T^0$ with $S \subseteq T^0 \subseteq E$ and the guaranteed properties in (Step 1). Then apply a ``degree -1 operation to $T^0$'' to obtain a complex $T^1$ with $T^0 \subseteq T^1 \subseteq E$. Then again apply a ``degree 0 operation to $T^1$ to obtain a $T^2$. We continue to use ``degree $k$ operations'' on the previously constructed complex in the following back and forth pattern on $k$:
$$0, \ \ \ -1,0,1, \ \ \ -2,-1,0,1,2, \ \ \ -3,-2,-1,0,1,2,3, \ \ \ \cdots$$
to build an increasing union of exact subcomplexes, $\{T^l\}$. Finally, set $P = \cup_{l \in \N} T^l$. We now verify that $P$ has the desired properties:
\begin{enumerate}
\item $S \subseteq P \subseteq E$ and each $P_n$ is $\gamma$-presented.

\noindent (Reason) The containments are clear and each $P_n = \cup_{l \in \N} (T^l)_n$ is $\gamma$-presented because all the $(T^l)_n$ are $\gamma$-presented and $|\N| < \gamma$. See~\cite[Proposition~1.16]{adamek-rosicky}.
\item $P$ is $G$-acyclic.

\noindent (Reason) $P$ is exact since it is a direct union of exact subcomplexes. Moreover each $Z_nP = \cup_{l \in \N} Z_n(T^l)$ must be a $G$-subobject of $P_n$ by Proposition~\ref{prop-direct limits} as the union contains a cofinal collection of $G$-monomorphisms.

\item $P_n \subseteq_G E_n$ for each $n$.

\noindent (Reason) Each $P_n = \cup_{l \in \N} (T^l)_n$ must be a $G$-subobject of $E_n$ because again, this union contains a cofinal collection of $G$-subobjects of $E_n$ by property (3) in (Step 1).
\end{enumerate}
\end{proof}

\begin{proposition}\label{prop-transfinite extensions from purity}
Let $\gamma$ be as in Setup~\ref{setup-gamma}.
Each $G$-acyclic complex is a transfinite $G$-extension of $\gamma$-presented $G$-acyclic complexes.
\end{proposition}

\begin{proof}
Suppose $E \neq 0$ is $G$-acyclic and use Lemma~\ref{lemma-finding G-exact subcomplexes} to find a nonzero $\gamma$-presented $G$-acyclic subcomplex $0 \neq P_0 \subseteq E$ which is a $G$-subobject in each degree. Then applying $\Hom_{\cat{G}}(G,-)$ to $P_0 \rightarrowtail E \twoheadrightarrow E/P_0$ leaves an exact sequence of complexes and it follows that $E/P_0$ is $G$-acyclic also. Assuming this complex is not zero
find another nonzero $\gamma$-presented $G$-acyclic complex $P_1/P_0 \subseteq E/P_0$ which is a $G$-subobject in each degree. Since $P_0 \subseteq E$ is a $G$-subobject in each degree, we get that $P_0 \subseteq P_1$ is also a $G$-subobject in each degree by Proposition~\ref{prop-G-purity properties}, part~(2). Then part~(3) of that same Proposition tells us that $P_1 \subseteq E$ is a $G$-subobject in each degree. Assuming $P_1 \neq E$, we continue to find an increasing sequence $0 \neq P_0 \subsetneq P_1 \subsetneq P_2 \subsetneq \cdots $ of $G$-acyclic subcomplexes of $E$ with each $P_n \subseteq E$ a $G$-subobject in each degree. Then set $P_{\omega} = \cup_{n<\omega} P_n$ and we see that $P_{\omega}$ is too a $G$-acyclic complex by Proposition~\ref{prop-direct limits} and also $P_{\omega} \subseteq E$ is a $G$-subobject in each degree, again by Proposition~\ref{prop-direct limits}. Therefore $E/P_{\omega}$ is also $G$-acyclic and we can continue with transfinite induction to end up with $E$ displayed as a transfinite $G$-extension of $\gamma$-presented $G$-acyclic complexes.
\end{proof}

\begin{theorem}\label{them-injective cotorsion pair for G-derived}
Let $\cat{G}$ be a Grothendieck category with a generator $G = \oplus_{i \in I} G_i$ where each $G_i$ is finitely presented.  Let $\class{W}$ be the class of all $G$-acyclic complexes. Then there is an injective cotorsion pair $(\class{W},\class{I})$ in $\cha{G}_G$. That is, it is a complete cotorsion pair in $\cha{G}_G$ for which $\class{W}$ is thick in $\cha{G}_G$ and $\class{W} \cap \class{I}$ coincides with the class of injective complexes in $\cha{G}_G$. We call the complexes in $\class{I}$ the \textbf{semi-$\boldsymbol{G}$-injective} complexes.
\end{theorem}

\begin{proof}
Let $\gamma$ be as in Setup~\ref{setup-gamma} and take $\class{S}$ to be a set of isomorphism representatives for the class of all $\gamma$-presented complexes in $\class{W}$. So everything in $\class{S}$ is a $G$-acyclic complex $S$ with each $S_n$ being $\gamma$-presented.
We will show that $\class{S}$ cogenerates $(\class{W},\class{I})$ in $\cha{G}_G$.
Recall that cotorsion pairs in $\cha{G}_G$ are with respect to $\GExt^1_{\cha{G}}$.
By Remark~\ref{remark on cotorsion pairs of complexes} which follows Proposition~\ref{prop-cotorsion pairs in the G-exact cat}, we know that $\class{S}$ cogenerates a complete cotorsion pair $(\leftperp{(\rightperp{\class{S}})}, \rightperp{\class{S}})$ in $\cha{G}_G$ where $\leftperp{(\rightperp{\class{S}})}$ consists precisely of retracts of transfinite $G$-extensions of $\class{S}$. We wish to show $\class{W} = \leftperp{(\rightperp{\class{S}})}$. But we already know that $\class{W}$ is thick in $\cha{G}_G$ by Lemma~\ref{lemma-properties G-exact} and closed under transfinite $G$-extensions by Corollary~\ref{cor-transfinite compositions and extensions of G-monos}. So $\class{W} \supseteq \leftperp{(\rightperp{\class{S}})}$. On the other hand, $\class{W} \subseteq \leftperp{(\rightperp{\class{S}})}$ was proved in Proposition~\ref{prop-transfinite extensions from purity}. So $(\class{W},\class{I})$ is a complete cotorsion pair where $\class{I} = \rightperp{\class{S}}$.

Since we already know $\class{W}$ is thick, all that is left is to show that $\class{W} \cap \class{I}$ coincides with the class of injective complexes in $\cha{G}_G$. But by the argument in~\cite[Proposition~3.3]{bravo-gillespie-hovey} it is enough to show that the injectives in $\cha{G}_G$ are contained in $\class{W}$. Since the injective complexes are precisely the contractible complexes with $G$-injective components by Lemma~\ref{lemma-complexes that are injective}, these are in $\class{W}$ by lemma~\ref{lemma-properties G-exact}.
\end{proof}

The following corollary now follows from the main result in~\cite{hovey}.

\begin{corollary}\label{cor-injective model for G-derived}
Let $\cat{G}$ be a Grothendieck category with a generator $G = \oplus_{i \in I} G_i$ where each $G_i$ is finitely presented. Then there is a model structure on $\cha{G}$ which we call the \textbf{$\boldsymbol{G}$-injective model structure} whose trivial objects are the $G$-acyclic complexes. The model structure satisfies the following:
\begin{enumerate}
\item The cofibrations are precisely the $G$-monomorphisms. That is, the chain maps which are $G$-monomorphisms in each degree.
\item The trivial cofibrations are the $G$-monomorphisms with $G$-acyclic cokernel.
\item The fibrations are the degreewise split epimorphisms whose kernel is a semi-$G$-injective complex.
\item The trivial fibrations are the split epimorphisms whose kernel is a $G$-injective complex.
\item The weak equivalences are the $G$-homology isomorphisms.
\item The model structure is cofibrantly generated. Sets of generating cofibrations and generating trivial cofibrations can be found using Proposition~\ref{prop-cotorsion pairs in the G-exact cat}.
\item The homotopy category is equivalent to $\class{D}(G)$, and this is a compactly generated triangulated category by Corollary~\ref{cor-projective model for G-derived}.
\end{enumerate}
\end{corollary}

\section{The recollement situations}\label{sec-recollement of Krause for G-derived}

Again, $\cat{G}$ is a locally finitely presentable Grothendieck category with generator $G = \oplus_{i \in I} G_i$ where each $G_i$ is finitely presented. Here we wish to prove the two recollement situations from Theorems~\ref{them-A} and~\ref{them-B} of the Introduction.

We will use the correspondence between injective (resp. projective) cotorsion pairs and recollements situations from~\cite{gillespie-recoll} and~\cite{gillespie-recoll2}. By definition, a cotorsion pair $(\class{P},\class{W})$ in $\cat{G}_G$ (or $\cha{G}_G$) is a \textbf{projective cotorsion pair} if it is complete, $\class{W}$ is $G$-thick, and if $\class{P} \cap \class{W}$ coincides with the class of $G$-projective objects. Since the category $\cat{G}_G$ has enough projectives this makes the triple $(\class{P},\class{W},\class{A})$, where $\class{A}$ represents the class of all objects, correspond to a model structure on $\cat{G}$ via Hovey's correspondence~\cite[Theorem~2.2]{hovey}. For example, the cotorsion pair of Theorem~\ref{them-projective model for G-derived} is a projective cotorsion pair in $\cha{G}_G$ and corresponds to the model structure of Corollary~\ref{cor-projective model for G-derived}. On the other hand, we showed in Proposition~\ref{prop-G-exact categories have enough injectives} that $\cat{G}_G$ (and so $\cha{G}_G$) also has enough injectives and so it also makes sense to speak of \textbf{injective cotorsion pairs} which are the dual. For example, the cotorsion pair of Theorem~\ref{them-injective cotorsion pair for G-derived} is an injective cotorsion pair in $\cha{G}_G$ and gave us the model structure of Corollary~\ref{cor-injective model for G-derived}.

\begin{proposition}\label{prop-G-INJ model struc}
Assume each $G_i$ is finitely presented. There is an injective model structure $(\class{W}_1,\class{F}_1)$ in $\cha{G}_G$ where $\class{F}_1$ is the class of all complexes of $G$-injective complexes.
\end{proposition}

\begin{proof}
From Proposition~\ref{prop-cotorsion pairs in the G-exact cat} and Remark~\ref{remark on cotorsion pairs of complexes} which follows it, we know that \emph{any} set of complexes cogenerates a complete cotorsion pair in $\cha{G}_G$. Here we let $\class{S}_1 = \{D^n(S) \,|\, S \in \class{S}\}$ where $\class{S}$ is the set in Proposition~\ref{prop-G-exact categories have enough injectives} which cogenerates the injective cotorsion pair  $(\class{A},\class{I})$ in $\cat{G}_G$. So $\class{I}$ is the class of $G$-injectives. By Lemma~\ref{lemma-disk adjunctions} we have $\GExt^1_{\cha{G}}(D^n(S),X) \cong \GExt^1_{\cat{G}}(S,X_n)$. It follows that $\rightperp{\class{S}_1} = \class{F}_1$ in $\cha{G}_G$.
So we get a complete cotorsion pair $(\class{W}_1,\class{F}_1)$ in $\cha{G}_G$ where $\class{F}_1$ is the class of all complexes of $G$-injective complexes.

To show it is an injective cotorsion pair in $\cha{G}_G$, we only need to show that $\class{W}_1$ is $G$-thick and contains the injectives. Note that for any complex $W$ and $F \in \class{F}_1$ we have $\GExt^1_{\cha{G}}(W,F) = \Ext^1_{dw}(W,F)$. So by Lemma~\ref{lemma-homcomplex-basic-lemma}, $W \in \class{W}_1$ if and only if $\homcomplex(W,F)$ is exact. So to see that $\class{W}_1$ is $G$-thick we consider a degreewise $G$-exact sequence of complexes $0 \xrightarrow{} X \xrightarrow{} Y \xrightarrow{} Z \xrightarrow{} 0$. Then as noted earlier, for any complex $F$ of $G$-injectives, applying $\homcomplex(-,F)$ will give us a short exact sequence $0 \xrightarrow{} \homcomplex(Z,F) \xrightarrow{} \homcomplex(Y,F) \xrightarrow{} \homcomplex(X,F) \xrightarrow{} 0$. So if two out of the three complexes are exact, then so is the third.  This proves thickness of $\class{W}_1$ in $\cha{G}_G$. If $I$ is an injective complex in $\cha{G}_G$, then by Lemma~\ref{lemma-complexes that are injective} it is a split exact complex with $G$-injective components. In particular, it is contractible. So for such an $I$ we have $\homcomplex(I,F)$ is exact for any $F \in \class{F}_1$.
\end{proof}

\begin{proposition}\label{prop-G-exact G-INJ model struc}
Assume each $G_i$ is finitely presented. There is an injective model structure $(\class{W}_2,\class{F}_2)$ in $\cha{G}_G$ where $\class{F}_2$ is the class of all $G$-acyclic complexes of $G$-injectives.
\end{proposition}

\begin{proof}
Take $\class{S}_1$ from the proof of Proposition~\ref{prop-G-INJ model struc} and let $\class{S}_2 = \class{S}_1 \cup \{S^n(G)\}$. We claim that $\rightperp{\class{S}_2} = \class{F}_2$ in $\cha{G}_G$. Indeed if $X \in \rightperp{\class{S}_2}$ then $X$ is a complex of $G$-injectives for which $0 = \GExt^1_{\cha{G}}(S^n(G),X) = \Ext^1_{dw}(S^n(G),X) = H_{n-1}\homcomplex(S^0(G),X) = H_{n-1}\Hom_{\cat{G}}(G,X)$. So $X$ is $G$-acyclic. Conversely, if $X$ is $G$-acyclic with $G$-injective components then $X \in \rightperp{\class{S}_2}$. So we get a complete cotorsion pair by again applying Proposition~\ref{prop-cotorsion pairs in the G-exact cat} and Remark~\ref{remark on cotorsion pairs of complexes} which follows it. The fact that $\class{W}_2$ is thick and contains the $G$-injective complexes follows just like in Proposition~\ref{prop-G-INJ model struc}.
\end{proof}

\begin{theorem}[Krause's recollement for $G$-derived categories]\label{them-krause recollement for G-derived}
Assume each $G_i$ is finitely presented.
Let $\class{D}(G)$ denote the $G$-derived category. Let $K_G(Inj)$ denote the homotopy category of all complexes of $G$-injectives. Let $K_{G\text{-ac}}(Inj)$ denote the homotopy category of all $G$-acyclic complexes of $G$-injectives. Then there is a recollement
\[
\xy
(-30,0)*+{K_{G\textnormal{-ac}}(Inj)};
(0,0)*+{K_G(Inj)};
(26,0)*+{\class{D}(G)};
{(-19,0) \ar (-10,0)};
{(-10,0) \ar@<0.5em> (-19,0)};
{(-10,0) \ar@<-0.5em> (-19,0)};
{(10,0) \ar (19,0)};
{(19,0) \ar@<0.5em> (10,0)};
{(19,0) \ar@<-0.5em> (10,0)};
\endxy
.\]
\end{theorem}

\begin{proof}
Take $(\class{W}_1,\class{F}_1)$ to be the injective cotorsion pair from Proposition~\ref{prop-G-INJ model struc}.
Take $(\class{W}_2,\class{F}_2)$ to be the injective cotorsion pair from Proposition~\ref{prop-G-exact G-INJ model struc}.
Take $(\class{W}_3,\class{F}_3) = (\class{W},\class{I})$ to be the semi-$G$-injective cotorsion pair from Theorem~\ref{them-injective cotorsion pair for G-derived}.
Since $\class{F}_2,\class{F}_3 \subseteq \class{F}_1$ and $\class{W}_3 \cap \class{F}_1 = \class{F}_2$ the result is automatic from~\cite[Theorem~3.4]{gillespie-recoll2}.
\end{proof}

\begin{theorem}[Verdier localization recollement for $G$-derived categories]\label{them-verdier recollement for G-derived}
Assume each $G_i$ is finitely presented. Let $\class{D}(G)$ denote the $G$-derived category. Let $K(\cat{G})$ denote the homotopy category of all chain complexes and let $K_{G\textnormal{-ac}}(\cat{G})$ denote the subcategory of all $G$-acyclic complexes. Then there is a recollement
\[
\begin{tikzpicture}[node distance=3.5cm]
\node (A) {$K_{G\textnormal{-ac}}(\cat{G})$};
\node (B) [right of=A] {$K(\cat{G})$};
\node (C) [right of=B] {$\class{D}(G)$};
\draw[<-,bend left=40] (A.20) to node[above]{\small E$(K\class{P},\class{W})$} (B.160);
\draw[->] (A) to node[above]{\small $I$} (B);
\draw[<-,bend right=40] (A.340) to node [below]{\small C$(\class{W},K\class{I})$} (B.200);
\draw[<-,bend left] (B.20) to node[above]{\small $\lambda = \text{C}(K\class{P},\class{W})$} (C.160);
\draw[->] (B) to node[above]{\small Q} (C);
\draw[<-,bend right] (B.340) to node [below]{\small $\rho = \text{E}(\class{W},K\class{I})$} (C.200);
\end{tikzpicture}
\]
Here,  $\class{W}$ is the class of $G$-acyclic complexes, and the complexes in $K\class{P}$ are the $G$-analog of Spaltenstein's K-projective complexes. The functor $\text{C}(K\class{P},\class{W})$ is the functor taking $X$ to its $K\class{P}$-precover since $(K\class{P},\class{W})$ turns out to be a complete cotorsion pair in $\cha{G}_{dw}$. Similarly $K\class{I}$ is analogous to the class of K-injective complexes and $\text{E}(\class{W},K\class{I})$ is the functor taking $X$ to its $K\class{I}$-preenvelope.
\end{theorem}

\begin{proof}
The basic idea is that the existence of the $G$-projective model $(\class{P},\class{W})$ of Section~\ref{sec-the G-derived category of a locally finitely presented Grothendieck category} provides a left adjoint to the inclusion $K_{G\textnormal{-ac}}(\cat{G}) \xrightarrow{} K(\cat{G})$, and in fact a colocalization sequence $K_{G\textnormal{-ac}}(\cat{G}) \xrightarrow{} K(\cat{G}) \xrightarrow{} \class{D}(G)$. On the other hand, the existence of the $G$-injective model $(\class{W},\class{I})$ of Section~\ref{sec-injective model struc for G-derived} provides a right adjoint to the inclusion $K_{G\textnormal{-ac}}(\cat{G}) \xrightarrow{} K(\cat{G})$, and in fact a localization sequence $K_{G\textnormal{-ac}}(\cat{G}) \xrightarrow{} K(\cat{G}) \xrightarrow{} \class{D}(G)$. Together this is a recollement. The formalization in terms of model structures follows immediately from work in~\cite[Section~6]{gillespie-recoll2}. The theory there is all written in terms of weakly idempotent complete exact categories, and so applies to our current setting. In full detail, we apply~\cite[Theorem~6.3]{gillespie-recoll2} to the $G$-injective model structure $(\class{W},\class{I})$ to obtain a Quillen equivalent model structure $(\class{W},K\class{I})$ in the exact category $\cha{G}_{dw}$ of chain complexes with degreewise split short exact sequences. The complexes in $K\class{I}$ are the $G$-analog of Spaltenstein's K-injective complexes and in fact are, by~\cite[Proposition~6.4]{gillespie-recoll2}, precisely the complexes that are chain homotopy equivalent to a semi-$G$-injective complex. The dual of~\cite[Theorem~6.3]{gillespie-recoll2} applied to the $G$-projective model structure $(\class{P},\class{W})$ gives us a similar model $(K\class{P},\class{W})$. All together $(K\class{P},\class{W},K\class{I})$ is \emph{localizing cotorsion triple} in the sense of~\cite[Section~4.1]{gillespie-recoll2} and so by~\cite[Corollary~4.5]{gillespie-recoll2} we obtain the recollement.
\end{proof}

\appendix

\section{$\lambda$-purity in Grothendieck categories}\label{appendix-lambda pure}  Every Grothendieck category $\cat{G}$ is locally presentable. This means there exists a regular cardinal $\lambda$ and a set $\class{S}$ of $\lambda$-presented objects such that every object of $\cat{G}$ is a $\lambda$-directed colimit of objects of $\class{S}$. In this case we say $\cat{G}$ is locally $\lambda$-presentable and it is true that for any regular cardinal $\lambda' > \lambda$, we have $\cat{G}$ is locally $\lambda'$-presentable as well. See~\cite[Theorem~1.20 and the Remark]{adamek-rosicky}.

Now following~\cite{adamek-rosicky}, a morphism $f$ is called \emph{$\lambda$-pure} if for each commutative diagram
$$\begin{CD}
A'       @>f'>>     B'       \\
@V u VV             @VV v V               \\
A       @>f>>  B \\
\end{CD}$$
with $A',B'$ $\lambda$-presented there is a map $t$ such that $u=tf'$. Assuming the category is locally $\lambda$-presentable we have from~\cite[Proposition~2.29]{adamek-rosicky} that a $\lambda$-pure morphism must be a monomorphism. In fact, they are characterized in~\cite[Proposition~2.30 and its Corollary]{adamek-rosicky} as being precisely the $\lambda$-directed colimits (in the category of morphisms) of split monomorphisms. Since Grothendieck categories are abelian we are lead naturally to speak instead of $\lambda$-pure short exact sequences, which we now characterize.

\begin{proposition}[$\lambda$-purity in Grothendieck categories]\label{prop-lambda pure short exact sequences}
Let $\cat{G}$ be a locally $\lambda$-presentable Grothendieck category and let $\class{E} : 0 \xrightarrow{} A \xrightarrow{f} B \xrightarrow{g} C \xrightarrow{} 0$ be a short exact sequence. Then the following are equivalent and characterize what we mean by saying \emph{$\class{E}$ is a $\lambda$-pure short exact sequence}.
\begin{enumerate}
\item $f$ is a $\lambda$-pure morphism.

\item $\Hom_{\cat{G}}(L,\class{E})$ is a short exact sequence of abelian groups for any $\lambda$-presented object $L$.

\item $\class{E}$ is a $\lambda$-directed limit of split short exact sequences $$\class{E}_i : 0 \xrightarrow{} A_i \xrightarrow{} B_i \xrightarrow{} C_i \xrightarrow{} 0 \ \ (i \in I).$$
\end{enumerate}

\end{proposition}

\begin{proof}
As already pointed out above, we have from~\cite[Proposition~2.30 and Corollary]{adamek-rosicky} that the $\lambda$-pure morphisms are precisely the $\lambda$-directed colimits of split monomorphisms. In particular, if $f$ is a $\lambda$-pure morphism, we get that the short exact sequence
$$\class{E} : 0 \xrightarrow{} A \xrightarrow{f} B \xrightarrow{g} C \xrightarrow{} 0$$ must be a $\lambda$-directed colimit of split short exact sequences
$$\class{E}_i : 0 \xrightarrow{} A_i \xrightarrow{} B_i \xrightarrow{} C_i \xrightarrow{} 0.$$ So (1) if and only if (3). But if (3) holds, then we clearly have that each $\Hom_{\class{G}}(L,\class{E}_i)$ is exact for any $L$. If $L$ is $\lambda$-presented then $\Hom_{\class{G}}(L,\class{E}) \cong \varinjlim \Hom_{\class{G}}(L,\class{E}_i)$ is also exact. So (3) implies (2).

Now we show (2) implies (3). Using that $\cat{G}$ is locally $\lambda$-presentable, write $C = \varinjlim C_i$ as a $\lambda$-directed colimit of $\lambda$-presented $C_i$. For each $\gamma_i : C_i \xrightarrow{} C$, form the pullback
$$\begin{CD}
\class{E}_i : \ \ \ 0 @>>> A       @>>>     B_i  @>>> C_i @>>> 0     \\
@. @|       @VVV      @VV \gamma_i V    @.           \\
\class{E} : \ \ \ 0 @>>> A       @>f>>  B @>g>> C @>>> 0 \\
\end{CD}$$
If (2) holds, then $\gamma_i$ lifts over $g$. This implies that $\class{E}_i$ splits. One can check that $\class{E} \cong \varinjlim \class{E}_i$.
\end{proof}

\begin{proposition}\label{prop-purity properties}
Let $\cat{G}$ be a locally $\lambda$-presentable Grothendieck category and $A \subseteq B \subseteq C$.
\begin{enumerate}
\item If $A \subseteq B$ is $\lambda$-pure and $B \subseteq C$ is $\lambda$-pure then $A \subseteq C$ is $\lambda$-pure.
\item If $A \subseteq C$ is $\lambda$-pure then $A \subseteq B$ is $\lambda$-pure.
\item If $A \subseteq C$ is $\lambda$-pure and $B/A \subseteq C/A$ is $\lambda$-pure, then $B \subseteq C$ is $\lambda$-pure.
\item $\lambda$-pure monomorphisms are closed under $\lambda$-directed colimits.
\end{enumerate}
\end{proposition}

\begin{proof}
(1) and (2) follow easy from the definition of $\lambda$-pure via the commutative diagram. For (3), let $L$ be $\lambda$-presented. All we need to check is that the map $\Hom_{\cat{G}}(L,C) \xrightarrow{} \Hom_{\cat{G}}(L,C/B)$ is an epimorphism. But this is just the composite $$\Hom_{\cat{G}}(L,C) \xrightarrow{} \Hom_{\cat{G}}(L,C/A) \xrightarrow{} \Hom_{\cat{G}}(L,(C/A)/(B/A)) \cong \Hom_{\cat{G}}(L,C/B),$$ and these are epimorphisms by hypothesis. Finally, a proof of (4) appears in~\cite[Proposition~2.30~(1)]{adamek-rosicky}.
\end{proof}

\section{Exact categories vs. proper classes}\label{appendix-proper classes}

We show here that if $\class{A}$ is an abelian category, an exact category $(\class{A},\class{E})$ in the sense of~\cite{quillen-algebraic K-theory} and~\cite{buhler-exact categories} is the same thing as a proper class of short exact sequences in the sense of~\cite[Chapter~XII.4]{homology} and~\cite{hovey}. See also the Historical Notes and Appendix~B of~\cite{buhler-exact categories} for the equivalence to Heller's axioms for an ``abelian class of short exact sequences''.

\begin{proposition}\label{prop-abelian exact categories are proper classes}
Let $\class{A}$ be an abelian category. Then $(\class{A},\class{E})$ is an exact category in the sense of~\cite{quillen-algebraic K-theory} if and only if $\class{E}$ is a proper class of short exact sequences in the sense of~\cite[Chapter~XII.4]{homology}.
\end{proposition}

\begin{proof}
Say $(\class{A},\class{E})$ is an exact category. We wish to see that $\class{E}$ is a proper class. The only thing that is not immediate from first definitions or properties of exact categories is Mac\,Lane's axiom (P-4), and the dual (P-4'). But abelian categories are weakly idempotent complete and so these follow from~\cite[Proposition~7.6]{buhler-exact categories} which states: whenever $gf$ is an admissible monomorphism (resp. epimorphism) then $f$ (resp. $g$) is an admissible monomorphism (resp. epimorphism).

On the other hand, say $\class{E}$ is a proper class in $\class{A}$. To see $(\class{A},\class{E})$ is an exact category we just need to check the pullback/pushout axioms. But any, say pullback, exists, and pulling back along an $\class{E}$-epimorphism $p$ yields a diagram:
$$\begin{CD}
0       @>>>    A    @>i'>> P @>p'>> C' @>>> 0  \\
@.            @|   @Vf' VV @VfVV     @.   \\
0       @>>>  A @>i>> B    @>p>> C @>>> 0\\
\end{CD}$$
Since $i$ is an $\class{E}$-monomorphism, so is $i = f'i'$. We wish to ``cancel'' $f'$ to conclude $i'$ is an $\class{E}$-monomorphism. However, axiom (P-4) of~\cite[Chapter~XII.4]{homology} only allows this when $f'$ is monic. But we now remedy this by imitating the argument that can be found within the proof of~\cite[XII.4 Theorem~4.3]{homology}. First, recall that the pullback $(P,f',p')$ can be constructed (see~\cite[XII.4 Theorem~1.1]{homology}) so that $P$ is the kernel in the left exact sequence $0 \xrightarrow{} P \xrightarrow{v} B \oplus C' \xrightarrow{p \pi_1 - f \pi_2} C$ and the maps $f'$ and $p'$ satisfy $f' = \pi_1 v$ and $p' = \pi_2 v$. We see that $$vi' = 1vi' = (i_1\pi_1 + i_2\pi_2)vi' = i_1(\pi_1v)i' + i_2(\pi_2v)i' = i_1f'i' + i_2p'i' = i_1i .$$
Since $i_1$ is an $\class{E}$-monomorphism by (P-2), we see that $i_1i$ is an $\class{E}$-monomorphism by (P-3). So $vi' = i_1i$ is an $\class{E}$-monomorphism, and by (P-4) we may now conclude $i'$ is an $\class{E}$-monomorphism.
\end{proof}

\end{document}